\documentclass[a4paper,reqno]{amsart}

\usepackage[utf8]{inputenc}
\usepackage[T1]{fontenc}
\usepackage{lmodern}
\usepackage[english]{babel}
\usepackage{microtype}

\usepackage{amsmath,amssymb,amsfonts,amsthm}
\usepackage{comment,mathtools,accents, upgreek}
\usepackage{mathrsfs}
\usepackage{aliascnt}
\usepackage{braket}
\usepackage{bm}
\usepackage{esint}

\newcommand{\nocontentsline}[3]{}
\newcommand{\tocless}[2]{\bgroup\let\addcontentsline=\nocontentsline#1{#2}\egroup}

\usepackage[a4paper,margin=3cm]{geometry}
\usepackage[citecolor=blue,colorlinks]{hyperref}

\usepackage{enumerate}
\usepackage{xcolor}

\makeatletter
\g@addto@macro\@floatboxreset\centering
\makeatother

\newcommand*{\minnefont}{\fontfamily{pzc}\selectfont}
\DeclareTextFontCommand{\minnetextcommand}{\minnefont\color{green}}

\newcommand*{\davidfont}{\fontfamily{lmtt}\selectfont}
\DeclareTextFontCommand{\davidcomment}{\davidfont\color{purple}}

\theoremstyle{plain}
\newtheorem*{theorem*}{Theorem}
\newtheorem{theorem}{Theorem}[section]
\newtheorem{lemma}[theorem]{Lemma}
\newtheorem{cor}[theorem]{Corollary}
\newtheorem{prop}[theorem]{Proposition}

\theoremstyle{definition}
\newtheorem{D}[theorem]{Definition} 
\newtheorem{rem}[theorem]{Remark}
\newtheorem{example}[theorem]{Example}

\mathtoolsset{showonlyrefs} 

\newcommand{\setN}{\mathbb{N}}

\newcommand{\setR}{\mathbb{R}}

\newcommand{\cT}{\mathcal{T}}

\newcommand{\eps}{\varepsilon}

\newcommand{\df}{\coloneqq}
\newcommand{\fd}{\eqqcolon}
\let\phi\varphi

\newcommand{\weakto}{\rightharpoonup}

\DeclareMathOperator{\Tr}{Tr}

\DeclareMathOperator{\sgn}{sgn}

\newcommand{\di}{\mathop{}\!\mathrm{d}}

\newcommand{\loc}{{\rm loc}}

\DeclareMathOperator{\supp}{supp}

\newcommand{\Leb}{\mathrm{Leb}}

\DeclareMathOperator{\Lip}{Lip}

\newcommand{\haus}{\mathscr{H}}
\newcommand{\leb}{\mathscr{L}}

\newcommand{\measrestr}{%
  \,\raisebox{-.127ex}{\reflectbox{\rotatebox[origin=br]{-90}{$\lnot$}}}\,%
}

\newcommand{\dist}{\mathsf{d}}

\DeclareMathOperator{\RCD}{RCD}
\DeclareMathOperator{\Span}{Span}

\DeclareMathOperator*{\esssup}{ess\,sup}

\newfont{\tmpf}{cmsy10 scaled 2500}

\newcommand{\R}{\ensuremath{\mathbb R}}

\DeclareMathOperator{\AMV}{AMV}
\DeclareMathOperator{\SAMV}{SAMV}

\def\un{\mathbf{1}}

\title[$\SAMV$ and $\AMV$ Laplacian in metric measure spaces]{Symmetrized and non-symmetrized Asymptotic Mean Value Laplacian in metric measure spaces}

\author{Andreas Minne}
\address{A.~Minne: KTH Royal Institute of Technology,  100 44 Stockholm,  Sweden.}
\email{minne@kth.se}

\author{David Tewodrose}
\address{D.~Tewodrose: Nantes Université, Laboratoire de Mathématiques Jean Leray, UMR CNRS 6629, 2 rue de la Houssinière
BP 92208
F-44322 Nantes Cedex 3, France.}
\email{David.Tewodrose@univ-nantes.fr}
\date{}

\begin{document}

\nocite{*} 

\maketitle

\begin{abstract}
The asymptotic mean value Laplacian --- AMV Laplacian --- extends the Laplace operator from $\mathbb{R}^n$ to metric measure spaces through limits of averaging integrals. The AMV Laplacian is however not a symmetric operator in general.  Therefore we consider a symmetric version of the AMV Laplacian, and focus lies on when the symmetric and non-symmetric AMV Laplacians coincide.  Besides Riemannian and 3D contact sub-Riemannian manifolds,  we show that they are identical on a large class of metric measure spaces, including locally Ahlfors regular spaces with suitably vanishing distortion.  In addition, we study the context of weighted domains of $\setR^n$ --- where the two operators typically differ --- and provide explicit formulae for these operators, including points where the weight vanishes.
\end{abstract}

\hfill


\hfill


\tableofcontents

\section{Introduction}

For any $u \in C^2(\setR^n)$ and $x \in \setR^n$,  the second-order Taylor expansion of $u$ at $x$ yields the identity
\[
\fint_{B_{r}(x)}u(y)-u(x) \di y=c_n r^{2}\Delta u(x)+o(r^{2}) 
\]
as $r \downarrow 0$, where $c_n \df 1/(2n+4).$ Building upon this elementary observation, we proposed in \cite{MT} a definition for the pointwise Laplacian of a locally integrable function $u$ defined on a general metric measure space $(X,\dist,\mu)$ by setting
\begin{equation}\label{eq:nonrigorous_old_def}
\Delta^{\dist}_{\mu} u(x)\coloneqq \lim\limits_{r \downarrow 0} \frac{1}{r^2} \fint_{B_r(x)}u(y) - u(x)\di \mu(y)
\end{equation}
for any $x \in X$ where the limit converges, see Definition \ref{def:AMV} for a rigorous definition. The quantity $\Delta^{\dist}_{\mu} u(x)$ is called the \emph{pointwise Asymptotic Mean Value} (AMV for short) \emph{Laplacian} of $u$ at $x$. We proved several results including maximum and comparison principles for the operator $\Delta^{\dist}_{\mu}$. 

However,  a disadvantage of $\Delta^{\dist}_{\mu}$ is that it --- in general --- is not a self-adjoint operator on $L^2(X,\mu)$, see \cite[Theorem 5.3 and Remark 5.4]{MT}.  To face this issue, we consider in this paper a \emph{symmetrized} version of $\Delta^{\dist}_{\mu}$, inspired by the symmetrized $r$-Laplacian introduced by Adamowicz,  Kijowski and Soultanis \cite{AKS_2},  which we define as
\[
\tilde{\Delta}^{\dist}_{\mu}u(x)\coloneqq\lim_{r\downarrow 0}\frac{1}{2r^2}\fint_{B_r(x)} (u(y) -u(x)) \left( 1 + \frac{\mu(B_r(x))}{\mu(B_r(y))} \right) \di \mu(y).
\]
We call $\tilde{\Delta}^{\dist}_{\mu}u(x)$ the \emph{pointwise Symmetrized Asymptotic Mean Value} (SAMV for short) \emph{Laplacian} of $u$ at $x$.  The Fubini theorem together with the symmetry of the metric show that
\begin{equation*}
\int_X (\tilde{\Delta}^{\dist}_{\mu}u) v \di \mu = \int_X u(\tilde{\Delta}^{\dist}_{\mu}v) \di \mu
\end{equation*}
for any $u,v \in L^2(X,\mu)$ for which $\tilde{\Delta}^{\dist}_{\mu}u$ and $\tilde{\Delta}^{\dist}_{\mu}v$ converge in $L^2(X,\mu)$; see Remark \ref{rk:sym}.

The primary goal of this article is to study for which instances the operators $\Delta^{\dist}_{\mu}$ and $\tilde{\Delta}^{\dist}_{\mu}$ coincide, and in which sense they do.  This coincidence depends on the infinitesimal interplay between the metric and the measure. For that reason,  the asymptotic behavior as $r \downarrow 0$ of the function 
\[
\delta_r\colon (x,y) \mapsto  1 - \frac{\mu(B_r(x))}{\mu(B_r(y))} 
\]
which we call the \textit{$r$-distortion} of the space (see Definition \ref{def:distortion}),
plays a critical role. Indeed,  the assumption
\[
 \sup_{y \in B_r(x)} | \delta_r(x,y) | =o(r^2)
\]
as $r \downarrow 0$ ensures that the pointwise equality $\Delta^{\dist}_{\mu}u(x) = \tilde{\Delta}^{\dist}_{\mu}u(x)$ holds, see Proposition \ref{prop1}.  On smooth Riemannian manifolds equipped with the canonical Riemannian distance and volume measure, this observation coupled with the classical volume expansion of asymptotically small balls easily yield that the $\AMV$ and $\SAMV$ Laplacians coincide almost everywhere, see Proposition \ref{prop:Riem}.  A similar expansion found by Barilari, Beschatnyi and Lerario \cite{BBL} implies that the same holds true on $3$-dimensional contact sub-Riemannian manifolds equipped with the Popp volume, see Proposition \ref{prop:subRiem}; the case of general sub-Riemannian manifolds would be worth a deeper investigation.

In our previous article, we also defined a weak $\AMV$ Laplacian \cite[Definition 5.5]{MT}.  We refine this notion in the present paper, and we additionally propose a definition of weak $\SAMV$ Laplacian: see Definition \ref{def:weak}.  In Proposition \ref{prop:extension_of_AKS}, we prove that on a locally Ahlfors regular space, if
\[
 \frac{1}{r} \left( \fint_{B_r(\cdot)} |\delta_r(\cdot,y)| \di \mu(y)\right)\di \mu(\cdot) \weakto 0
\]
as $r \downarrow 0$, in the sense of weak convergence of measures against compactly supported continuous functions, then the weak $\AMV$ and $\SAMV$ Laplacians of any Lipschitz function coincide if one of the two exists. This allows us to remove the curvature assumption made in \cite{AKS_2} where the coincidence between the weak $\AMV$ and $\SAMV$ Laplacians was established for non-collapsed $\RCD(K,N)$ spaces with vanishing metric-measure boundary, see Remark \ref{rem:extension_of_AKS}.

The secondary goal of this article is to derive concrete formulae for the pointwise $\AMV$ and $\SAMV$ Laplacians in weighted domains of $\setR^n$, that is to say metric measure spaces of the form $(\Omega,\dist,\mu)$ where $\Omega$ is a domain of $\setR^n$, $\dist$ is a distance on $\Omega$ and $\mu$ is absolutely continuous with respect to the restriction of the Lebesgue measure $\leb^n$ to $\Omega$, with Radon--Nikodym derivative $w$. In this context, there are two different cases depending on whether $w$ is strictly positive or zero at the point $x \in \Omega$. 

If $w(x)>0$ we assume that $w \in C^1(\Omega)$ and that $\dist$ satisfies a mild symmetry assumption, \eqref{eq:symmetricball}. We show that in this case, if the matrix of limits of average second-moments 
\[
M(x) \df \lim\limits_{r \downarrow 0} \left( \fint_{B_r(x)} \frac{(y-x)_i(y-x)_j}{r^2} \di y \right)_{1\le i,j \le n}
\]
exists, then for any $u \in C^2(\Omega)$,
\begin{equation}\label{eq:Delta}
\Delta_{\mu}^{\dist}u(x)	 =\frac{1}{2}\Tr(M(x)\nabla^2 u(x))+\frac{1}{w(x)}\langle\nabla w(x),M(x)\nabla u(x)\rangle.
\end{equation}
If in addition one has
\begin{equation}\label{eq:z/r}
\frac{1}{r} \fint_{B_r(x)} \left| 1 - \frac{\leb^n(B_r(x))}{\leb^n(B_r(y))}\right| \di y  \to 0
\end{equation}
as $r \downarrow 0$, then
\begin{equation}\label{eq:tildeDelta}
\tilde{\Delta}_{\mu}^{\dist}u(x)=\frac{1}{2w(x)}\Tr(M(x)\nabla(w\nabla u)(x)),
\end{equation}
see Proposition \ref{prop:sym_Hausdorff_distortion}.  In case $\dist$ is associated with a norm,
equality \eqref{eq:Delta} has already been established in \cite{AKS_1} in an $L^p_{\text{loc}}$ sense (see Section~\ref{sec:preliminaries}); in this case the limit matrix $M(x)$ always exists and is independent of $x$. For the Euclidean distance, \eqref{eq:Delta} was proven in \cite[Proposition~2.3]{MT}. It is likely that \eqref{eq:tildeDelta} also holds in an $L^p_{\text{loc}}$ sense under an appropriate $L^p_{\text{loc}}$ replacement of \eqref{eq:z/r}, but we do not investigate this question in this paper.

If $w(x)=0$, we restrict our study to the case where $\dist$ is associated with a norm, hence we can assume without loss of generality that $x = 0_n \in \Omega$. We also assume that $w$ is locally integrable, in which case the assumption $w(x)=0$ must be understood in a Lebesgue point sense.  In Proposition~\ref{prop:as_even_weights}, we show that under the infinitesimal evenness property
\[
 \text{$\leb^n$-}\esssup\limits_{x \in B_r(0_n)}  |w(x)-w(-x)|=o\left( r \fint_{B_r(0_n)} w \di \leb^n \right) \qquad \text{as $r \downarrow 0$},
\]
and the existence of a weak limit measure
\[
\nu \df \lim\limits_{r \downarrow 0}\frac{w(r\, \cdot)r^n}{\mu(B_r(x))} \leb^n \measrestr B_1(0_n),
\]
setting $M_\nu\coloneqq \left( \int_{B_1(0_n)} y_i y_j \di \nu(y) \right)_{1\le i,j \le n}$, we get
\[
\Delta^{\dist}_{\mu}u(0_n)=\frac{1}{2} \Tr (M_\nu \nabla^2u)(0_n).
\]
Under a natural comparability assumption in the spirit of \cite{NaorTao,Aldaz1,Aldaz2} (see Definition~\ref{def:comparability}), we get an analogous result in the case of the $\SAMV$ Laplacian, see Proposition  \ref{prop:sym_as_even_weights}. 

Let us conclude this section with a last comment: Proposition \ref{prop:positive_weights} shows that on weighted Euclidean domains the $\SAMV$ Laplacian coincides with the usual weighted Laplacian $\Delta u + \langle\nabla \ln w,\nabla u\rangle$ (also known as drifted Laplacian, or $f$-Laplacian, or Witten Laplacian,  see e.g.~\cite{munteanu2011smooth,colbois2015eigenvalues} and the references therein) while the $\AMV$ does not.  This suggests that $\tilde{\Delta}^{\dist}_{\mu}$ is the natural $\AMV$ Laplacian to consider in the future,  in accordance with \cite{AKS_2} where local minimizers of the Korevaar--Schoen energy with target $\setR$ are identified as those Sobolev maps having zero weak $\SAMV$ Laplacian. 

\smallskip\noindent
\textbf{Acknowledgements.} The first author is grateful for
the support of the Knut and Alice Wallenberg Foundation, Project KAW 2015.0380.  The second author is supported by Laboratoire de Mathématiques Jean Leray via the project Centre Henri Lebesgue and Fédération de recherche Mathématiques de Pays de Loire via the project Ambition Lebesgue Loire. The authors thank Giorgio Stefani and Luca Rizzi for pointing out the reference \cite{BBL}. They are also both grateful for  precious remarks made by the anonymous referees.  The second author thanks Gilles Carron for inspiring discussions and Luca Rizzi for his invitation at SISSA Trieste — funded by the European Research Council (ERC) under the European Union’s Horizon 2020 research and innovation programme (grant agreement No. 945655) — where helpful discussions took place.

\section{Preliminaries}\label{sec:preliminaries}

Throughout the paper,  we write $\AMV$ as a shorthand for ``Asymptotic Mean Value'' and $\SAMV$ for ``Symmetrized Asymptotic Mean Value''.  We keep a positive integer $n$ fixed and we denote by $0_n$ the origin of the vector space $\setR^n$,  by $\dist_{e}$ the Euclidean distance on $\setR^n$,  by $\omega_n$ the Lebesgue measure of the unit Euclidean ball and by $\sigma_{n-1}$ the total surface measure of the unit Euclidean sphere.

We say that a triple $(X,\dist,\mu)$ is a metric measure space if $(X,\dist)$ is a metric space and $\mu$ is a fully supported Borel measure on $(X,\dist)$ such that $0<\mu(B_r(x))<+\infty$ for any $x \in X$ and $r>0$.  In case $X$ is an open subset of a space equipped with a norm $\|\cdot\|$ and $\dist$ is the canonical distance associated with $\|\cdot\|$, we write $(X,\|\cdot\|,\mu)$.

Let $(X,\dist,\mu)$ be a metric measure space. We use classical notation to denote function spaces on $X$, like $C(X)$ for the set of continuous functions, $\Lip(X)$ for the set of Lipschitz functions,  and so on.  We set $\overline{\setR}\coloneqq\setR\cup\{\pm \infty\}$.  We shall often identify a $\mu$-measurable function $f : X \to \overline{\setR}$ with its equivalence class under $\mu$-a.e.~equality. For any $p>0$, we let $L^p(X,\mu)$ be the set of $p$-integrable functions, that is, $\mu$-measurable functions $f : X \to \overline{\setR}$ such that $\int_X |f|^p \di \mu < +\infty$. We also let $L^p_{\text{loc}}(X,\mu)$ be the set of locally $p$-integrable functions, that is,  $\mu$-measurable functions $f : X \to \overline{\setR}$ such that any $x \in X$ admits a neighborhood $V$ such that $\int_V |f|^p \di \mu < +\infty$.  We let $L^\infty(X,\mu)$ be the set of essentially bounded $\mu$-measurable functions on $X$, and $L_\text{loc}^\infty(X,\mu)$ be the set of locally essentially bounded $\mu$-measurable functions on $X$, that is, the $\mu$-measurable functions $f : X \to \overline{\setR}$ such that $\mu$-a.e.~$x \in X$ admits a neighborhood $V$ such that the essential supremum of $f$ on $V$ is bounded from above.  In case $X$ is an open subset of $\setR^n$ and $\mu$ is the Lebesgue measure, we simply write $L^p(X)$, $L^p_{\text{\loc}}(X)$, $L^\infty(X)$, and  $L^\infty_{\text{\loc}}(X)$ respectively. If $A$ is a bounded Borel subset of $X$,  we write $\un_A$ for the characteristic function of $A$, and for any locally integrable function $u:X \to \R$, we set
\[
\fint_A u \di \mu \coloneqq \frac{1}{\mu(A)} \int_A u \di \mu. 
\] 

Assume that $(X,\dist)$ is locally compact. Under this assumption,  we shall often make use of compact sets.  We use the subscript $c$ to denote the restriction of a function space to compactly supported functions: for instance,  by $C_c(X)$ we mean the set of compactly supported continuous functions.  We recall that a Radon measure $\nu$ on $(X,\dist)$ is a Borel measure which is locally finite and inner regular, meaning that any $x \in X$ admits a neighborhood $V$ such that $\nu(V)<+\infty$, and any measurable subset $A \subset X$ satisfies $\mu(A)=\sup\{\mu(K) \, : \, K \subset A \, \, \text{compact}\}$, respectively.  The Riesz--Markov--Kakutani theorem states that the space of Radon measures $\mathrm{Rad}(X)$ is the topological dual of $C_c(X)$. Therefore, we equip $\mathrm{Rad}(X)$ with the associated weak topology, and we denote by $\weakto$ the convergence in this topology.  Lastly, for any $f \in L^1_{\text{loc}}(X,\mu)$,  we denote by $f \mu$ the Radon measure defined by
\[
(f \mu)(A) \df \int_A f \di \mu
\]
for any measurable $A \subset X$.

\hfill

\paragraph{\textbf{Lebesgue points.}}

Let $(X,\dist,\mu)$ be a metric measure space. For any $u \in L^1_{\text{loc}}(X,\mu)$, we recall that $x \in X$ is a Lebesgue point of $u$ if there exists some real number $u^{*}(x)$ such that
\begin{equation}\label{eq:Lebesguepoints}
\lim\limits_{r\downarrow 0}\fint_{B_r(x)} |u(y)-u^*(x)|\di \mu(y) = 0.
\end{equation}
We write $\Leb(u)$ for the set of Lebesgue points of $u$. It is easily checked that neither the set $\Leb(u)$ nor the real number $u^*(x)$ for any $x \in \Leb(u)$ depend on the choice of a representative in the equivalence class of $u$.  Moreover, the following hold.
\begin{enumerate}
\item If $(X,\dist)$ is locally compact, then for any $u \in C(X)$ a simple argument based on the Heine--Cantor theorem shows that $\Leb(u)=X$ and $u^*(x)=u(x)$ for any $x \in X$.
\item If $(X,\dist,\mu)$ is infinitesimally doubling, meaning that
\begin{equation}\label{eq:infinitesimaldoubling}
\limsup_{r \downarrow 0} \frac{\mu(B_{2r}(x))}{\mu(B_r(x))}<+\infty
\end{equation}
for $\mu$-a.e.~$x \in X$, then $\mu(X\backslash \Leb(u))=0$ for any $u \in L^1_{\text{loc}}(X,\mu)$ and $u^*(x)=u(x)$ for $\mu$-a.e.~$x \in X$. This follows from the Lebesgue Differentiation Theorem for infinitesimally doubling metric measure spaces: see \cite[Theorem 3.4.3 and page 77]{HKST}.
\end{enumerate}

\hfill

\paragraph{\textbf{Comparability conditions.}}

We now introduce a couple of definitions which originate at least from \cite{NaorTao}. They have also appeared in a different way in  \cite{Aldaz1, Aldaz2}.

\begin{D}[Comparability conditions]\label{def:comparability}
Let $(X,\dist,\mu)$ be a metric measure space.

\begin{enumerate}
\item\label{1} We say that $\mu$ satisfies a \emph{comparability condition on $A \subset X$} if for any $x \in A$ there exist $r_0=r_0(x),C=C(x)>0$ such that for any $r \in (0,r_0)$, for $\mu$-a.e.~$y \in B_r(x)$,
\begin{equation}\label{eq:LCC}
\mu(B_r(x)) \le C \mu(B_r(y)).
\end{equation}
In case $A=\{x\}$ for some $x \in X$ we often say that $\mu$ satisfies a \emph{comparability condition at $x$}.
\item\label{2} We say that $\mu$ satisfies a \emph{locally uniform comparability condition} if any $x_0 \in X$ admits a neighborhood $V$ and two constants $r_0=r_0(x_0),C=C(x_0)>0$ such that the inequality \eqref{eq:LCC} holds for $\mu$-a.e.~$x \in V$, for any $r \in (0,r_0)$ and for $\mu$-a.e.~$y \in B_r(x)$.
\item\label{3} We say that $\mu$ satisfies a \emph{uniform comparability condition} if there exist $r_0,C>0$ such that for any $x \in X$ the inequality \eqref{eq:LCC} holds for any $r \in (0,r_0)$ and $\mu$-a.e.~$y \in B_r(x)$.
\end{enumerate}
\end{D}

Observe that: $\eqref{3} \Rightarrow \eqref{2} \Rightarrow  \eqref{1}$.

\begin{rem}\label{rem:locally_compact}
If $(X,\dist)$ is locally compact, then the neighborhoods in the previous definition can be chosen compact. In particular, if $\mu$ satisfies a locally uniform comparability condition, it is not difficult to show that for any compact set $K \subset X$ there exists $r_K,C_K>0$ such that 
\[
\mu(B_r(x)) \le C_K \mu(B_r(y))
\]
for $\mu$-a.e.~$x \in K$, any $r \in (0,r_K)$ and $\mu$-a.e.~$y \in B_r(x)$.
\end{rem}

\begin{example}
A metric measure space is doubling at scale $r_0>0$ if there exists a constant $C\ge 1$ such that
\[
\mu(B_{2r}(x)) \le C \mu(B_r(x))
\]
for any $x \in X$ and $r \in (0,r_0/2)$.  Any such space trivially satisfies a uniform comparability condition \eqref{3}, since for any $y \in B_r(x)$ where $x,r$ are as above,
\[
\mu(B_r(x)) \le \mu(B_{2r}(y)) \le C \mu(B_r(y)).
\]
\end{example}

\begin{example}\label{ex:Ahlfors}
For any $Q>0$,  we say that a metric measure space $(X,\dist,\mu)$ is Ahlfors $Q$-regular if there exist $C\ge 1$ and $r_0\in(0,+\infty]$ such that
\begin{equation}\label{eq:Ahlfors}
C^{-1} r^Q \le \mu(B_r(x)) \le C r^Q
\end{equation}
for any $x \in X$ and $r \in (0,r_0)$.  Non-trivial examples including fractal-type ones may be found in \cite{BourdonPajot,Laakso,KleinerSchioppa}, see also \cite{Barlow} where they are called fractional metric measure spaces.   Ahlfors regular spaces are obviously doubling, hence they satisfy a uniform comparability condition \eqref{3}. Moreover, we say that $(X,\dist,\mu)$ is locally Ahlfors $Q$-regular if any $x_0 \in X$ admits a neighborhood $V$ and two constants $r_0=r_0(x_0),C=C(x_0)>0$ such that
\[
C^{-1} r^Q \le \mu(B_r(x)) \le C r^Q
\]
for any $x \in V$ and $r \in (0,r_0)$.  In this case $\mu$ trivially satisfies a locally uniform comparability condition \eqref{2}.
\end{example}

\begin{example} Let us provide an example of a space which, according to the terminology introduced in Definition \ref{def:comparability}, satisfies \eqref{1} at any point but does not satisfy \eqref{3}. For any $n \in \setN\backslash \{0\}$,  let $\mathcal{C}_n$ be the circle of center $x_n\df (3n,0)$ with radius $1/n$ and $y_{n,1},\ldots,y_{n,n}$ be distinct points on $\mathcal{C}_n$.  Join each $y_{n,i}$ with $x_n$ by the geodesic segment between them and consider the set $X \subset \setR^2$ obtained as the union of all these segments and $[1,+\infty) \times \{0\}$.  Equip $X$ with the length distance $\dist$ induced by the Euclidean one on $\setR^2$ and with the Borel measure $\mu$ defined as the sum of the $1$-dimensional Hausdorff measure and Dirac masses at each $x_n$, $y_{n,i}$. Then $\mu(B_{3/(2n)}(x_n)) \le 1 + n + 3/2$ while $\mu(B_{3/(2n)}(y_{n,i})) = 5/2 + 1/n$ for any $n$ and any $i$, hence $(X,\dist,\mu)$ cannot satisfy \eqref{3}.
\end{example}

\hfill

\paragraph{\textbf{Averaging operator.}}

We recall now the definition and some properties of the averaging operator.

\begin{D}[Averaging operator]\label{def:adjoint}
Let $(X,\dist,\mu)$ be a metric measure space.  For any $r>0$, the \emph{averaging operator $A_r$} is defined by setting, for all $f \in L^1_{\text{loc}}(X,\mu)$ and $\mu$-a.e.~$x \in X$,
\[
A_rf(x)\coloneqq\fint_{B_r(x)} f \di \mu.
\]
\end{D}

\begin{rem}\label{rem}
Obviously $\sup_{r>0}\|A_r\|_{L^\infty(X,\mu)\to L^\infty(X,\mu)} \le 1$.  In addition,  if $f \in L^\infty_{\text{loc}}(X,\mu)$, then any $x \in X$ admits a neighborhood $V$ and a radius $r_0=r_0(x)$ such that $\|A_r f\|_{L^\infty(V,\mu)}<+\infty$ for any $r \in (0,r_0)$. Indeed, if $W$ is a neighborhood of $x$ such that $\|f\|_{L^\infty(W,\mu)}<+\infty$, then there exists $r_0>$ such that $B_{2r_0}(x) \subset W$, and for any $r \in (0,r_0)$ and $\mu$-a.e.~$y \in B_{r_0}(x)$,
\[
|A_rf(y)| \le \fint_{B_r(y)} |f(z)| \di \mu(z) \le \|f\|_{L^\infty(B_r(y),\mu)} \le \|f\|_{L^\infty(W,\mu)}<+\infty,
\]
hence we can take $V=B_{r_0}(x)$.
\end{rem}

\begin{lemma}\label{lemma_averaging}
Let $(X,\dist,\mu)$ be a metric measure space and $p \in [1,+\infty)$. 
\begin{enumerate}
\item Assume that $\mu$ satisfies a uniform comparability condition \eqref{3} with parameters $r_0,C$.  Then the operator $A_r$ is bounded from $L^p(X,\mu)$ into itself, with norm at most $C^{1/p}$, for any $r \in (0,r_0)$.
\item Assume that $\mu$ satisfies a locally uniform comparability condition. Then for any $f \in L_{\text{loc}}^p(X,\mu)$ and $x \in X$, there exist $r_0=r_0(x,f),C=C(x,f)>0$ and a neighborhood $V$ of $x$ also depending of $f$, such that for any $r \in (0,r_0)$,
\[
\int_V |A_r f|^p \di \mu < C \int_V  |f|^p \di \mu < +\infty.
\]
\end{enumerate}
\end{lemma}

\begin{proof}
$\textit{(1)}$ If $\mu$ satisfies a uniform comparability condition with parameters $r_0,C$, then for any $x \in X$ and $r \in (0,r_0)$,
\[
a_r(x) \coloneqq \int_{B_r(x)} \frac{\di \mu(y)}{\mu(B_r(y))} \le C.
\]
This implies boundedness of $A_r$ from $L^1(X,\mu)$ to itself with norm at most $C$, see \cite[Theorem 3.3]{Aldaz1} for instance. Then the Jensen inequality implies the case $p >1$, see e.g.~ \cite[Theorem 2.9]{Aldaz2}. 

 $\textit{(2)}$ Consider $x \in X$ and $f \in L^p_{\text{loc}}(X,\mu)$. By \eqref{eq:2}, there exist $r_0=r_0(x,f),C=C(x,f)>0$ and a neighborhood $V$ of $x$ also depending of $f$ such that $\int_V |f|^p \di \mu< +\infty$, $B_{2r_0}(x) \subset V$ and $\mu(B_r(y))\le C \mu(B_r(z))$ for $\mu$-a.e.~$y \in V$, any $r \in (0,r_0)$ and $\mu$-a.e.~$z \in B_{r}(y)$. The Hölder inequality yields 
 \[
\left( \fint_{B_r(z)} |f(y)| \di \mu(y) \right)^{p} \le \fint_{B_r(z)} |f(y)|^p \di \mu(y). 
 \]
 Then
\begin{align*}
\int_V |A_r f|^p \di \mu \le \int_V \fint_{B_r(z)} |f(y)|^p \di \mu(y) \di \mu(z) & =  \int_V \int_V 1_{B_r(z)}(y) \frac{|f(y)|^p}{\mu(B_r(z))} \di \mu(y) \di \mu(z)\\
& \le \int_V  |f(y)|^p  \int_{B_r(y)} \frac{\di \mu(z)}{\mu(B_r(z))} \di \mu(y)\\
& \le C \int_V  |f(y)|^p \di \mu(y).
\end{align*}
\end{proof}

\begin{rem}[Adjoint of $A_r$]
Assume that a uniform comparability condition holds with parameters $r_0,C$. Then for any $r \in (0,r_0)$, the operator $A_r^{*}$ defined by 
\[
A_r^{*} f(x)\coloneqq\int_{B_r(x)} \frac{f(y)\di \mu(y)}{\mu(B_r(y))}
 \]
for $\mu$-a.e.~$x \in X$,  is bounded from $L^p(X,\mu)$ into itself for any $p \ge 1$.  Indeed,
\[
A_r^{*} f(x) = \fint_{B_r(x)} \frac{f(y)\mu(B_r(x))\di \mu(y)}{\mu(B_r(y))} \le C \fint_{B_r(x)} f(y)\di \mu(y) = CA_r f(x).
\]
Moreover, the Fubini theorem and the symmetry of the distance easily imply that 
\[
A_r^{*} f(x)\coloneqq\int_{B_r(x)} \frac{f(y)\di \mu(y)}{\mu(B_r(y))}
 \]
for $\mu$-a.e.~$x \in X$, for all $f \in L^2(X,\mu)$. This means that $A_r^*$ is the adjoint of $A_r$ with respect to the $L^2$ scalar product. More generally, we have that for any $p,q \ge 1$ such that $1/p+1/q=1$, one has
\[
\int_X (A_rg)h \di \mu = \int_X g (A_r^*h) \di \mu
\]
for all $g \in L^p(X,\mu)$ and $h \in L^q(X,\mu)$.
\end{rem}

\paragraph{\textbf{Distortion.}}

We finally introduce a couple of notions which helps us formalizing the infinitesimal relation between the metric and the measure of a metric measure space. Compare with \cite{KLP}.

\begin{D}\label{def:distortion}
Let $(X,\dist,\mu)$ be a metric measure space.  For any $r>0$, we define:
\begin{enumerate}
\item\label{11} for any $x,y \in X$,
\[
\delta_r (x,y) \df 1-\frac{\mu(B_r(x))}{\mu(B_r(y))} \, ,
\]
\item\label{22} for any $x \in X$,
\[
z_r(x) \coloneqq \fint_{B_r(x)} |\delta_r(x,y)| \di \mu(y) \, .
\]
\end{enumerate}
We say that $\delta_r$ is the  \emph{$r$-distortion function} of $(X,\dist,\mu)$.
 
 \begin{rem}
In \cite{KLP}, the authors considered the $r$-deviation functions
\[
v_r(x) \df 1 - \frac{\mu(B_r(x))}{\omega_n r^n}
\]
which measure in a very rough way how much a metric measure space differs from the Euclidean $\mathbb{R}^n$.  These deviation functions may behave quite differently from the distortion ones.  For instance, consider a smooth Riemannian manifold equipped with its canonical Riemannian distance and volume measure. At an interior point $x$ of such a space,
\[
v_r(x) = c S(x) r^2 + O(r^4) \quad \text{as $r \downarrow 0$,}
\]
where $S(x)$ is the scalar curvature at $x$ and $c$ is a dimensional constant (see \eqref{eq:scalar}).  However, we show in the proof of Proposition \ref{prop:Riem} that if $y=y(r) \in B_r(x)$ for any $r>0$, then
\[
\delta_r(x,y) = c(S(y)-S(x))r^2 + O(r^4) \quad \text{as $r \downarrow 0$,}
\]
where $O(r^4)$ does not depend on $y$.  In particular, if the scalar curvature is constant but non-zero, then $v_r(x)  =  c r^2 + O(r^4)$ but $\delta_r(x,y) = O(r^4)$.
 \end{rem}
\end{D}

\section{Definitions}\label{sec:definitions}

In this section, we refine the definitions of AMV Laplacian proposed in \cite{MT,AKS_1} and we put forward similar definitions for the SAMV Laplacian.

\hfill

\paragraph{\textbf{Pointwise AMV and SAMV Laplacians.}}
In order to define the pointwise $\AMV$ Laplacian of a locally integrable function $u$ over a metric measure space $(X,\dist,\mu)$,  in \cite[Definition 1.1]{MT} we made the convention of choosing the representative $\tilde{u}$ of $u$ defined by
\[
\tilde{u}(x) \df \lim\limits_{r \downarrow 0} \fint_{B_r(x)} u(y) \di \mu(y)
\]
for any $x \in X$ where this limit exist and then to set
\begin{equation}\label{eq:old_def}
\Delta^{\dist}_{\mu}u(x) \df \frac{1}{r^2}\fint_{B_r(x)} u(y) -\tilde{u}(x)\di \mu(y).
\end{equation}
However, this definition has a couple of disadvantages which we point out in Remark \ref{rk:issues}.  Therefore, we propose a new definition, formalized with the aid of Lebesgue points, along with a similar definition for the $\SAMV$ Laplacian.

\begin{D}[Pointwise $\AMV$/$\SAMV$ Laplacian]\label{def:AMV}
Let $(X,\dist,\mu)$ be a metric measure space and $u \in L^1_{\text{loc}}(X,\mu)$ be given.

\begin{enumerate}
\item For any $x \in \Leb(u)$ and $r>0$, we set
\[
\Delta^{\dist}_{\mu,r}u(x)\coloneqq\frac{1}{r^2}\fint_{B_r(x)} u(y) -u^*(x)\di \mu(y).
\]
If there exists a real number $\Delta^{\dist}_{\mu} u (x)$ such that
\begin{equation}\label{eq:new_def}
\lim\limits_{r \downarrow 0} \Delta^{\dist}_{\mu,r} u(x) = \Delta^{\dist}_{\mu}  u(x),
\end{equation}
then we say that $u$ admits a \emph{pointwise AMV Laplacian at $x$}, and we call $\Delta^{\dist}_{\mu}  u(x)$ the pointwise $\AMV$ Laplacian of $u$ at $x$.\\

\item For any $x \in \Leb(u)$ and $r>0$, we set
\[
\tilde{\Delta}^{\dist}_{\mu,r}u(x)\coloneqq\frac{1}{2r^2}\fint_{B_r(x)} (u(y) -u^*(x)) \left( 1 + \frac{\mu(B_r(x))}{\mu(B_r(y))} \right) \di \mu(y).
\]
If there exists a real number  $\tilde{\Delta}^{\dist}_{\mu} u(x)$ such that
\begin{equation}\label{eq:new_def_sym}
\lim\limits_{r \downarrow 0} \tilde{\Delta}^{\dist}_{\mu,r}u(x) = \tilde{\Delta}^{\dist}_{\mu} u(x),
\end{equation}
then we say that $u$ admits a \emph{pointwise SAMV Laplacian at $x$}, and we call $\tilde{\Delta}^{\dist}_{\mu}  u(x)$ the pointwise $\SAMV$ Laplacian of $u$ at $x$.
\end{enumerate}
\end{D}

\begin{rem}\label{rk:issues}
Let us explain why we favour \eqref{eq:new_def} over our previous definition \eqref{eq:old_def}. A first reason is that unlike \eqref{eq:old_def} this new definition does not require to choose any representative in the equivalent class of $u$. A second reason is the following: If $x \in \Leb(u)$ then $\tilde{u}(x)$ exists and is equal to $u^*(x)$. However, it may happen that $\tilde{u}(x)$ exists while $x \notin \Leb(u)$. For instance, let $u \in L^1(\setR)$ be the equivalent class of the sign function $\setR \backslash \{0\} \ni x \mapsto x/|x|$. Equip $\setR$ with the natural distance associated with the absolute value. Then $\tilde{u}(0)$ exists and is equal to $0$ while $0 \notin \Leb(u)$. Therefore, according to \eqref{eq:old_def} the function $u$ is pointwise AMV harmonic on $\setR$, while \eqref{eq:new_def} yields that $\Delta_\mu^\dist u (0)$ is not defined and $u$ is pointwise AMV harmonic on $\setR\backslash \{0\}$ only: this is in better agreement with the distributional Laplacian of $u$ being equal to a multiple of the derivative of the Dirac distribution in $0$, as mentioned in the last sentence of \cite{MT}.
\end{rem}

\paragraph{\textbf{Strong AMV/SAMV Laplacian.}} In \cite{AKS_1} the authors proposed a definition of strongly $\AMV$-harmonic functions by appealing on compact sets.  In order to work in a more general setting, we define here a strong $\AMV$ Laplacian and a strong $\SAMV$ Laplacian without involving compact sets.  Let $(X,\dist,\mu)$ be a metric measure space. Consider $u \in L^\infty_{\text{loc}}(X,\mu)$. According to Remark \ref{rem}, any $x \in X$ admits a neighborhood $V$ and a radius $r_0>0$ such that 
\[
\|\Delta^{\dist}_{\mu,r}u\|_{L^\infty(B_{r_0}(x),\mu)} \le \frac{2\|u\|_{L^\infty(V,\mu)}}{r^2}
\]
for any $r \in (0,r_0)$.  This observation guarantees the well-posedness of the following definition.

\begin{D}[Strong $\AMV$ Laplacian]
Let $(X,\dist,\mu)$ be a metric measure space. We say that $u \in L^\infty_{\text{loc}}(X,\mu)$ admits a strong $\AMV$-Laplacian if there exists a function $v \in L^\infty_{\text{loc}}(X,\mu)$ such that any $x \in X$ admits a neighborhood $V$ such that
\[
\lim\limits_{r \downarrow 0} \|\Delta^{\dist}_{\mu,r} u -v\|_{L^\infty(V,\mu)} = 0,
\]
in which case we say that $\Delta^{\dist}_{\mu} u \coloneqq v$ is the strong $\AMV$-Laplacian of $u$.
\end{D}

Let us provide a analogous definition in the symmetrized case. To this aim, we assume that $\mu$ satisfies a local comparability condition.  Under this condition, for any $u \in L^\infty_{\text{loc}}(X,\mu)$ and  $x \in X$, there exist $C,r_0>0$ and a neighborhood $V$ of $x$ such that
\[
\|\tilde{\Delta}^{\dist}_{\mu,r} u\|_{L^\infty(B_{r_0}(x))} \le \frac{2C\|u\|_{L^\infty(V,\mu)}}{r^2}
\]
for any $r\in(0,r_0)$. 

\begin{D}[Strong $\SAMV$ Laplacian]
Let $(X,\dist,\mu)$ be a metric measure space. Assume that $\mu$ satisfies a local comparability condition. We say that $u \in L^\infty_{\text{loc}}(X,\mu)$ admits a strong $\SAMV$-Laplacian if there exists a function $v \in L^\infty_{\text{loc}}(X,\mu)$ such that any $x \in X$ admits a neighborhood $V$ such that
\[
\lim\limits_{r \downarrow 0} \|\tilde{\Delta}^{\dist}_{\mu,r} u - v \|_{L^\infty(V,\mu)} = 0,
\]
in which case we say that $\tilde{\Delta}^{\dist}_{\mu} u\coloneqq v$ is the strong $\SAMV$-Laplacian of $u$.
\end{D}

\paragraph{\textbf{$L^p/L^p_{\text{loc}}$-AMV/SAMV Laplacian.}}  Let us consider $p\in [1,+\infty)$. Recall that Lemma \ref{lemma_averaging} ensures that under a uniform comparability condition, the operator $A_r$ is bounded from $L^p$ to itself for any sufficiently small $r$.  As a consequence, also $\Delta_{\mu,r}^\dist$ is bounded from $L^p$ to itself. Then the following definition is well-posed.

\begin{D}[$L^p$ $\AMV$/$\SAMV$ Laplacian]
Let $(X,\dist,\mu)$ be a metric measure space such that $\mu$ satisfies a uniform comparability condition. Consider $u \in L^1_{\text{loc}}(X,\mu)$ and $p\in [1,+\infty)$. 
\begin{enumerate}
\item We say that $u$ admits an $L^p$-\emph{AMV Laplacian} if there exists a function $v \in L^p(X,\mu)$ such that
\[
\lim\limits_{r \downarrow 0} \|\Delta^{\dist}_{\mu,r} u - v \|_{L^p(X,\mu)} = 0,
\]
in which case we say that $\Delta^{\dist}_{\mu}  u \coloneqq v$ is the $L^p$-$\AMV$ Laplacian of $u$.\\

\item We say that $u$ admits an \emph{$L^p$-SAMV Laplacian} if there exists a function $v \in L^p(X,\mu)$ such that 
\[
\lim\limits_{r \downarrow 0} \|\tilde{\Delta}^{\dist}_{\mu,r} u - v \|_{L^p(X,\mu)} = 0,
\]
in which case we say that $\tilde{\Delta}^{\dist}_{\mu}  u \coloneqq v$ is the $L^p$-$\SAMV$ Laplacian of $u$.
\end{enumerate}
\end{D}

\begin{rem}\label{rk:sym}
Let $u,v \in L^2(X,\mu)$ be both admitting an $L^2$-$\SAMV$ Laplacian.   Since 
\[
\tilde{\Delta}^{\dist}_{\mu,r} u(x) = \int_X k_r(x,y) (u(y)-u(x))\di \mu(y)
\]
for any $r>0$ and $\mu$-a.e.~$x \in X$, where
\[
k_r(x,y) \df \frac{1}{2r^2} \un_{[0,r)}(\dist(x,y))  \left( \frac{1}{\mu(B_r(x))} + \frac{1}{\mu(B_r(y))}  \right),
\]
the Fubini theorem implies
\[
\int_X (\tilde{\Delta}^{\dist}_{\mu,r}u) v \di \mu = \int_X u(\tilde{\Delta}^{\dist}_{\mu,r}v) \di \mu,
\]
hence from letting $r \downarrow 0$ we obtain
\begin{equation}\label{eq:sym}
\int_X (\tilde{\Delta}^{\dist}_{\mu}u) v \di \mu = \int_X u(\tilde{\Delta}^{\dist}_{\mu}v) \di \mu.
\end{equation}
\end{rem}

Let us now provide a definition of $L^p_{\text{loc}}$-AMV Laplacian and $L^p_{\text{loc}}$-SAMV Laplacian. In this case we assume a locally uniform comparability condition.

\begin{D}[$L^p_{\text{loc}}$ $\AMV$/$\SAMV$ Laplacian]
Let $(X,\dist,\mu)$ be a metric measure space such that $\mu$ satisfies a locally uniform comparability condition. Consider $u \in L^1_{\text{loc}}(X,\mu)$ and $p\in [1,+\infty)$.
\begin{enumerate}
\item We say that $u$ admits an $L^p_{\text{loc}}$-\emph{AMV Laplacian} if there exists a function $v \in L^p_{\text{loc}}(X,\mu)$ such that for any $x \in X$ there exists a neighborhood $V$  of $x$ such that
\[
\lim\limits_{r \downarrow 0} \|\Delta^{\dist}_{\mu,r} u - v \|_{L^p(V,\mu)} = 0,
\]
in which case we say that $\Delta^{\dist}_{\mu}  u \coloneqq v$ is the $L^p_{\text{loc}}$-$\AMV$ Laplacian of $u$.\\

\item We say that $u$ admits an \emph{$L^p_{\text{loc}}$-SAMV Laplacian} if there exists a function $v \in L^p_{\text{loc}}(X,\mu)$ such that for any $x \in X$ there exists a neighborhood $V$ of $x$ such that
\[
\lim\limits_{r \downarrow 0} \|\tilde{\Delta}^{\dist}_{\mu,r} u - v \|_{L^p(V,\mu)} = 0,
\]
in which case we say that $\tilde{\Delta}^{\dist}_{\mu}  u \coloneqq v$ is the $L^p_{\text{loc}}$-$\SAMV$ Laplacian of $u$.
\end{enumerate}
\end{D}

\begin{rem}
By Lemma \ref{lemma_averaging}, the locally uniform comparability condition guarantees that for any $x\in X$ there exist a neighborhood $V$ of $x$ and a constant $r_0$ such that $\Delta^{\dist}_{\mu,r} u \in L^p(V,\mu)$ for any $r \in (0,r_0)$. 
\end{rem}

\begin{rem}
The Hölder inequality implies that if $u$ admits an $L^p_{\text{loc}}$-$\AMV$ Laplacian, then $u$ admits an $L^{q}_{\text{loc}}$-$\AMV$ Laplacian for any $q \ge p$ which coincides $\mu$-a.e.~with the $L^p_{\text{loc}}$-$\AMV$ Laplacian. This remark holds in the $\SAMV$ case too.
\end{rem}

\hfill

\paragraph{\textbf{Weak AMV/SAMV Laplacian.}}  Let us now define a notion of weak $\AMV$ Laplacian and weak $\SAMV$ Laplacian. We focus on locally compact metric measure spaces $(X,\dist,\mu)$ such that $\mu$ satisfies a locally uniform comparability condition.  Under these assumptions,  Remark \ref{rem:locally_compact} implies that if $K$ is the support of a function $\phi \in C_c(X)$, then there exists $r_K>$ such that for any $r \in (0,r_K)$,
\[
\left| \int_K  \Delta^\dist_{\mu,r}\phi \di \mu \right| + \left| \int_K  \tilde{\Delta}^\dist_{\mu,r}\phi \di \mu \right| < +\infty.
\]
Then the following definition is well-posed.

\begin{D}[Weak $\AMV$/$\SAMV$ Laplacian]\label{def:weak}
Let $(X,\dist,\mu)$ be a locally compact metric measure space such that $\mu$ satisfies a locally uniform comparability condition. Let $u \in L^1_{\text{loc}}(X,\mu)$ be given.
\begin{enumerate}
\item We say that $u$ admits a \emph{weak Asymptotic Mean Value Laplacian} (weak $\AMV$ Laplacian for short) if there exists a Radon measure $\bold{\Delta}^\dist_{\mu}u$ such that
\[
(\Delta^\dist_{\mu,r}  u) \,  \mu \weakto  \bold{\Delta}^\dist_{\mu} u
\]
as $r \downarrow 0$, in which case we say that $\bold{\Delta}^\dist_{\mu}u$ is the weak $\AMV$ Laplacian of $u$.
\item We say that $u$ admits a \emph{weak Symmetrized Asymptotic Mean Value Laplacian} (weak $\SAMV$ Laplacian for short) if there exists a Radon measure $\bold{\tilde{\Delta}}^\dist_{\mu}u$ such that
\[
(\tilde{\Delta}^\dist_{\mu,r} u ) \,  \mu \weakto \bold{\tilde{\Delta}}^\dist_{\mu}u
\]
as $r \downarrow 0$, in which case we say that $ \bold{\tilde{\Delta}}^\dist_{\mu}u$ is the weak $\SAMV$ Laplacian of $u$.
\end{enumerate}
\end{D}

\begin{rem}
	A corresponding definition for the distributional AMV/SAMV Laplacian can be made for smooth manifolds by considering weak convergence in the dual of $C^\infty_c(X)$ instead of the dual of $C_c(X)$.
\end{rem}

We point out that if $u$ admits a weak $\SAMV$ Laplacian and $\phi \in C_c(X)$ admits an $L^2$-$\SAMV$ Laplacian, it follows from \eqref{eq:sym} that
\[
\int_X \phi \di \bold{\tilde{\Delta}}^\dist_{\mu}u = \int_X (\tilde{\Delta}^\dist_{\mu,r} \phi) u \di \mu.
\]

The connection between the weak and $L_{\text{loc}}^{p}$ Laplacians is the following.

\begin{prop}
Let $(X,\dist,\mu)$ be a locally compact metric measure space such that $\mu$ satisfies a locally uniform comparability condition.  Assume that $u \in L^1_{\text{loc}}(X,\mu)$  admits an $L_{\text{loc}}^{p}$-$\AMV$ Laplacian $\Delta_{\mu}^{\dist}u$ for some $p\in[1,\infty]$. Then $u$ admits a weak $\AMV$  Laplacian which is absolutely continuous with respect to $\mu$ with Radon--Nikodym derivative equal to $\Delta_{\mu}^{\dist}u$. The same holds with $\AMV$ replaced by $\SAMV$.
\end{prop}

\begin{proof}
Consider $\phi\in C_{c}(X)$ with support $K$. Let $q$ be the conjugate exponent of $p$, i.e.~$1/p+1/q=1$. Then for any sufficiently small $r>0$,
\begin{align*}
\left|\int_{X}\phi(\Delta_{\mu,r}^{\dist}u-\Delta_{\mu}^{\dist}u)\di \mu\right|	& \le\int_{K}\left|\phi\right|\left|\Delta_{\mu,r}^{\dist}u-\Delta_{\mu}^{\dist}u\right|\di \mu\\
& \le\left\Vert \phi\right\Vert _{L^{q}(K,\mu)}\left\Vert \Delta_{\mu,r}^{\dist}u-\Delta_{\mu}^{\dist}u\right\Vert _{L^{p}(K,\mu)}.
\end{align*}
The last term tends to zero as $r\downarrow 0$, hence $\lim_{r\downarrow 0}\int_{X}\phi\Delta_{\mu,r}^{\dist}u\,\di\mu=\int_{X}\phi\Delta_{\mu}^{\dist}u\,\di\mu$. The $\SAMV$ case is proved the same way.
\end{proof}

Some functions admit a distributional or weak $\AMV$ Laplacian but no $L_{\text{loc}}^{p}$-$\AMV$ Laplacian. Here is an example.

\begin{example}
Consider $(\mathbb{R},\dist_{\text{e}},\leb^1)$ and $u(x)=x/|x|$,
defined to be zero for $x=0$. Then an easy computation shows that
\[
\Delta_{\leb^1,r}^{\dist_{\text{e}}}u(x)=\begin{cases}
\frac{x-r\sgn x}{r^{3}}, & r>|x|,\\
0, & r\le|x|.
\end{cases}
\]
Let $V$ be any open neighborhood around the origin on which $v$ is $p$-integrable.
Then there exists an $\epsilon>0$ such that $(-\epsilon,\epsilon)\subseteq V$,
and we can assume without loss of generality that $3r<\epsilon$.
Then, for $p\in(1,+\infty)$,
\begin{align*}
\lVert\Delta_{\leb^1,r}^{\dist_{\text{e}}} u-v\rVert_{L^{p}(V)} & \ge\lVert\Delta_{\leb^1,r}^{\dist_{\text{e}}} u-v\rVert_{L^{p}([-\epsilon,\epsilon])}\\
 & \ge\lVert\Delta_{\leb^1,r}^{\dist_{\text{e}}} u\rVert_{L^{p}([-\epsilon,\epsilon])}-\lVert v\rVert_{L^{p}([-\epsilon,\epsilon])}\\
 & \ge\frac{1}{r^{3}}\left(\int_{2r}^{\epsilon}(x-r)^{p}dx\right)^{1/p}-\lVert v\rVert_{L^{p}([-\epsilon,\epsilon])}\\
 & \ge\frac{1}{r^{3}}\left((\epsilon-2r)(2r-r)^{p}\right)^{1/p}-\lVert v\rVert_{L^{p}([-\epsilon,\epsilon])}\\
 & \ge\frac{r^{1/p}\cdot r}{r^{3}}-\lVert v\rVert_{L^{p}([-\epsilon,\epsilon])}\to+\infty,\qquad r\downarrow 0.
\end{align*}
This shows that $u$ does not admit an $L_{\text{loc}}^{p}$-AMV Laplacian
for any $p\in(1,+\infty)$. This result trivially extends to
$p=\infty$. Instead, $u$ has a distributional $\AMV$ Laplacian which can be shown to be the distributional derivative of the Dirac delta at the origin divided by three \cite[p.~21]{MT}. 
\end{example}

\section{Equality between the $\AMV$ and $\SAMV$ Laplacians}

In this section, we study some contexts where the equality $\Delta^{\dist}_{\mu}=\tilde{\Delta}^{\dist}_{\mu}$ holds in a suitable sense. 

\hfill

\paragraph{\textbf{Topological groups.}}
We recall that a topological group is a set $G$ equipped with a group law $\cdot$ and a topology $\cT$ with respect to which the maps $(g,h) \mapsto g \cdot h$ and $g \mapsto g^{-1}$ are continuous.  A left-invariant distance $\dist$ on a topological group $(G,\cdot,\cT)$ is a distance on $G$ which induces the same topology as $\cT$ and such that
\[
\dist(g \cdot x, g \cdot y) = \dist(x,y)
\]
for any $g,x,y \in G$.  When a topological group $(G,\cdot,\cT)$ is locally compact, it admits a unique --- up to a positive multiplicative constant --- non-zero Borel measure $\mu$, called Haar measure, which is finite on compact sets, quasi-regular, and left-invariant in the sense that
\[
\mu(g\cdot A)=\mu(A)
\] for any $g \in G$ and $A \in \cT$.

\begin{prop}
Let $(G,\cdot,\cT,\mu)$ be a locally compact topological group equipped with a Haar measure $\mu$ and metrized by a left-invariant distance $\dist$. Then the following hold:
\begin{enumerate}
\item For any $u \in L^1_{\text{loc}}(G,\mu)$ and any $x \in \Leb(u)$,  one has that $\Delta^\dist_{\mu,r}u(x)$ converges if and only if $\tilde{\Delta}^\dist_{\mu,r} u(x)$ does,  in which case 
\begin{equation}\label{eq:1}
\Delta^{\dist}_{\mu}u(x)=\tilde{\Delta}^{\dist}_{\mu}u(x).
\end{equation}
\item For any $p \in [1,+\infty)$,  a function $u \in L^p(G,\mu)$ admits a $L^p_{\text{loc}}$-$\AMV$ Laplacian if and only if it admits a $L^p_{\text{loc}}$-$\SAMV$ Laplacian, in which case
\begin{equation}\label{eq:2}
\Delta^{\dist}_{\mu}u=\tilde{\Delta}^{\dist}_{\mu}u \qquad \text{$\mu$-a.e.}
\end{equation}
\end{enumerate}
\end{prop}

\begin{proof}
For any $x \in G$, $r>0$ and $y \in B_r(x)$, since $\dist$ is left-invariant,
\[
y\cdot x^{-1} \cdot B_r(x) = B_r(y),
\]
and since $\mu$ is left-invariant, this implies that
\begin{equation}\label{eq:trivial}
\mu(B_r(x)) = \mu(B_r(y)).
\end{equation}
Then for any $u \in L^1(G,\mu)$ and any $x \in \Leb(u)$, one has $\Delta^{\dist}_{\mu,r}u(x)=\tilde{\Delta}^{\dist}_{\mu,r}u(x)$ for any $r>0$,  from which \eqref{eq:1} follows as $r \downarrow 0$.  Moreover,  \eqref{eq:trivial} ensures that $\mu$ trivially satisfies a uniform comparability condition.
\end{proof}

\hfill

\paragraph{\textbf{Riemannian and sub-Riemannian manifolds}} We provide in the next proposition a general criterion to get that $\Delta^\dist_{\mu}$ and $\tilde{\Delta}^\dist_{\mu}$ coincide. We phrase this result in terms of the $r$-distortion function $\delta_r (x,y)$ introduced in Definition \ref{def:distortion}.

\begin{prop}\label{prop1}
Let $(X,\dist,\mu)$ be a metric measure space.  For any $x \in X$ and $r>0$, set
\begin{equation}\label{eq:deviation_infty}
\eps_x(r)\coloneqq \sup_{y \in B_r(x)} |\delta_r(x,y)|. 
\end{equation}
Let $u \in L^1_{\text{loc}}(X,\mu)$ be given. Then for any $x\in \Leb(u)$ such that
\begin{equation}\label{eq:condition}
\eps_x(r)=O(r^2),\qquad r \downarrow 0,
\end{equation}
it holds that
\begin{equation}\label{eq:concl}
\lim\limits_{r \downarrow 0} |\Delta^\dist_{\mu,r}u(x) - \tilde{\Delta}^\dist_{\mu,r} u(x)| = 0.
\end{equation}
In particular, if $x \in \Leb(u)$ satisfies \eqref{eq:condition}, then $\Delta^\dist_{\mu,r}u(x)$ converges if and only if $\tilde{\Delta}^\dist_{\mu,r} u(x)$ does,  in which case $\Delta^\dist_{\mu}u(x)= \tilde{\Delta}^\dist_{\mu} u(x)$.
\end{prop}

\begin{proof}
For any $x \in \Leb(u)$ and $r>0$, 
\begin{equation}\label{eq:distortion}
\Delta^\dist_{\mu,r}u(x) - \tilde{\Delta}^\dist_{\mu,r} u(x) = \frac{1}{2r^2} \fint_{B_r(x)} \delta_r(x,y) (u(y) - u^*(x)) \di \mu(y).
\end{equation}
Then
\[
|\Delta^\dist_{\mu,r}u(x) - \tilde{\Delta}^\dist_{\mu,r} u(x)| \le  \frac{\eps_x(r)}{2r^2}  \fint_{B_r(x)} |u(y)-u^*(x)| \di \mu(y).
\] The conclusion follows since $\fint_{B_r(x)} |u(y)-u^*(x)| \di \mu(y) \to 0$ as $r \downarrow 0$. 
\end{proof}

Thanks to Proposition \ref{prop1} we can immediately show that $\Delta^\dist_{\mu}$ and $\tilde{\Delta}^\dist_{\mu}$ coincide on regular enough Riemannian manifolds. Indeed, let $M$ be a smooth $n$-dimensional manifold equipped with a $C^2$ Riemannian metric.  Let $\mu$ be the associated Riemannian volume measure. As is well-known (see e.g.~\cite{GHL}),  the scalar curvature of $(M,g)$ is a continous function $S_g : M \to \setR$ such that
\begin{equation}\label{eq:scalar}
\frac{\mu(B_r(x))}{\omega_n r^n} = 1 - c_n S_g(x) r^2 + O(r^4) \quad \text{as $r \downarrow 0$}
\end{equation}
for any $x \in M$, where $c_n\df(n+2)^{-1}6^{-1}$ and $O(r^4)$ depends only on curvature terms of $g$ at $x$ --- in particular,  if $K$ is a compact subset of $M$, then $O(r^4)$ can be made independant of $x \in K$.

\begin{prop}\label{prop:Riem}
Let $M$ be a smooth manifold equipped with a $C^2$ Riemannian metric $g$. Let $\mu$ be the associated Riemannian volume measure. Let $u \in L^1_{\text{loc}}(M,\mu)$ be given. Then for $\mu$-a.e.~$x \in M$,
\[
\lim\limits_{r \downarrow 0} |\Delta_{\mu,r}^{\dist}u(x) - \tilde{\Delta}_{\mu,r}^{\dist} u(x)| = 0.
\]
In particular, for $\mu$-a.e.~$x \in M$,  one has that $\Delta^\dist_{\mu,r}u(x)$ converges if and only if $\tilde{\Delta}^\dist_{\mu,r} u(x)$ does,  in which case $\Delta^\dist_{\mu}u(x)= \tilde{\Delta}^\dist_{\mu} u(x)$.
\end{prop}

\begin{proof}
Let $n$ be the dimension of $M$. For any $x \in M$ and $y \in B_r(x)$,  thanks to \eqref{eq:scalar} we find that
\begin{align*}
\frac{\mu(B_r(y))}{\mu(B_r(x))} & = \frac{1-c_nS_g(y)r^2 + O(r^4)}{1-c_nS_g(x)r^2 + O(r^4)}\\
& = (1-c_nS_g(y)r^2 + O(r^4))(1+c_nS_g(x)r^2 + O(r^4))\\
& = 1 - c_n(S_g(y)-S_g(x))r^2 + O(r^4)
\end{align*} 
as $r \downarrow 0$, where we choose $O(r^4)$ uniform over points in the closed ball $\overline{B}_1(x)$. Since $S_g$ is continuous it is uniformly contiuous on $\overline{B}_1(x)$ hence it admits a non-decreasing modulus of continuity $\omega$. Therefore,
\[
\frac{\eps_x(r)}{r^2} \le c_n \left( \sup_{y \in B_r(x)} |S_g(y)-S_g(x)| \right)  + O(r^2) \le  c_n \omega(r)  + O(r^2)
\]
as $r \downarrow 0$.  This shows that \eqref{eq:condition} holds at any $x \in M$.

Let us show now that $(M,g)$ satisfies the infinitesimally doubling condition \eqref{eq:infinitesimaldoubling}. Thanks to \eqref{eq:scalar}, for any $x \in M$,
\begin{align*}
\frac{\mu(B_{2r}(x))}{\mu(B_r(x))} & = 2^n \frac{\mu(B_{2r}(x))}{\omega_n (2r)^n}\frac{\omega_n r^n}{\mu(B_r(x))} =  2^n \,\frac{1 - c_n S_g(z) (2r)^2 + O(r^4)}{1 - c_n S_g(z) r^2 + O(r^4)}\,,
\end{align*}
as $r \downarrow 0$, where we choose $O(r^4)$ uniform over points in the closed ball $\overline{B}_1(x)$, hence
\[
\lim_{r \downarrow 0} \frac{\mu(B_{2r}(x))}{\mu(B_r(x))} = 2^n.
\]

From \eqref{eq:infinitesimaldoubling} we get that $\mu$-a.e.~point $x \in M$ is a Lebesgue point of $u$, hence the result follows from Proposition \ref{prop1}.
\end{proof}

Let us pass now to the context of sub-Riemannian manifolds. We recall that a sub-Riemannian structure on a manifold $M$ is a pair $(D, g)$ where $D$ is a bracket generating subbundle of the tangent bundle $TM$, i.e.~Lie brackets of vector fields tangent to $D$ span the full tangent bundle, and $g$ is a smooth metric defined on $D$. Such a structure yields a well-defined distance $\dist$ called the sub-Riemannian (or Carnot--Carath\'eodory) distance. More precisely, $\dist(p, q)$ is the infimum of the length of Lipschitz curves tangent to $D$ (also called horizontal curves) joining two points $p$ and $q$. Here the length of the curve is computed with respect to the metric $g$.  We refer to \cite{AgrachevBarilariBoscain} for a complete introduction to sub-Riemannian geometry.

In this context,  the analysis made in \cite{BBL} yields the following.

\begin{prop}\label{prop:subRiem}
Let $(M,D,g)$ be a $3$-dimensional, contact sub-Riemannian manifold equipped with the Popp volume $\mu$. Let $u \in L^1_{\text{loc}}(M,\mu)$ be given. Then for $\mu$-a.e.~$x \in M$,  it holds that
\[
\lim\limits_{r \downarrow 0} |\Delta_{\mu,r}^{\dist}u(x) - \tilde{\Delta}_{\mu,r}^{\dist} u(x)| = 0.
\]
In particular, for $\mu$-a.e.~$x \in M$,  one has that $\Delta^\dist_{\mu,r}u(x)$ converges if and only if $\tilde{\Delta}^\dist_{\mu,r} u(x)$ does,  in which case $\Delta^\dist_{\mu}u(x)= \tilde{\Delta}^\dist_{\mu} u(x)$.
\end{prop}

\begin{proof}
From \cite[Theorem 1]{BBL} we know the following for any $x \in M$,
\begin{equation}\label{eq:subriem}
\frac{\mu(B_r(x))}{c_o r^4} = 1 - c_1 \kappa(x) r^2 + O(r^3) \quad \text{as $r \downarrow 0$},
\end{equation}
where $c_o$ and $c_1$ are positive constants and $\kappa$ depends smoothly of $x$. Then the proof of Proposition~\ref{prop:Riem} carries over and yields the result.
\end{proof}
 
\hfill

\paragraph{\textbf{Spaces with suitably vanishing distortion}} 

We investigate now the case of spaces where the distortion functions $\delta_r$ satisfy more subtle assumptions. We begin with the following elementary lemma. Recall that $z_r(x)$ is defined in Definition \ref{def:distortion} as the average of $|\delta_r(x,y)|$.

\begin{lemma}
Let $(X,\dist,\mu)$ be a locally compact metric measure space. Then for any $u \in \Lip(X)$, $\phi \in C_c(X)$, and $r>0$,
\[
\left| \int_X \phi (\Delta^{\dist}_{\mu,r}u-\tilde{\Delta}^{\dist}_{\mu,r} u) \di \mu\right|  \le \frac{\Lip(u)}{2}\int_X |\phi| \frac{z_r}{r}  \di \mu.
\]
\end{lemma}

\begin{proof}
The result follows from a direct computation:
\begin{align*}
\left| \int_X \phi (\Delta^{\dist}_{\mu,r}u-\tilde{\Delta}^{\dist}_{\mu,r} u) \di \mu\right| & \le \frac{1}{2}\int_X |\phi(x)|  \fint_{B_r(x)} |\delta_r(x,y)| \frac{|u(y)-u(x)|}{r^2}\di \mu(y) \di \mu(x)\\
& \le \frac{\Lip(u)}{2}\int_X |\phi(x)|  \fint_{B_r(x)} |\delta_r(x,y)| \frac{\dist(x,y)}{r^2}\di \mu(y) \di \mu(x)\\
& \le \frac{\Lip(u)}{2}\int_X |\phi(x)|  \fint_{B_r(x)} \frac{ |\delta_r(x,y)| }{r}\di \mu(y) \di \mu(x)\\
& = \frac{\Lip(u)}{2}\int_X |\phi(x)| \frac{z_r(x)}{r}  \di \mu(x).
\end{align*}
\end{proof}
As an immediate corollary,  we obtain the following result by letting $r \downarrow 0$.

\begin{cor}\label{cor:main}
Let $(X,\dist,\mu)$ be a locally compact metric measure space such that $z_r \in L^1_\text{loc}(X,\mu)$ for all $r>0$ small enough. Assume that
\begin{equation}
\frac{z_r}{r}  \mu \weakto 0 \qquad \text{as $r\downarrow 0$.}
\end{equation}
Then $u \in \Lip(X)$ admits a weak $\AMV$ Laplacian if and only if $u$ admits a weak $\SAMV$ Laplacian, in which case $\bold{\Delta}^{\dist}_\mu u=\bold{\tilde{\Delta}}^{\dist}_\mu u$.
\end{cor}

\begin{rem}\label{rem:implies(1)}
If $\mu$ satisfies a locally uniform comparability condition, then $z_r \in L^1_\text{loc}(X,\mu)$. Indeed, this assumption and the local compactness of $(X,\dist)$ imply that there exist $r_0,C>0$ such that for any $x \in X$, $r \in (0,r_0)$ and any compact neighborhood $K$ of $x$,
\[
\int_K z_r \di \mu \le (1+C) \mu\left(\cup_{\bar{x} \in K}B_r(\bar{x})\right)<+\infty.
\]
\end{rem}

We are going to apply Corollary \ref{cor:main} in the context of locally Ahlfors regular spaces (see Example~\ref{ex:Ahlfors} for the definition).  Let $Q$ be a fixed positive number.   We set
\[
\omega_Q \df \frac{\pi^{Q/2}}{\Gamma(Q/2+1)}
\]
where $\Gamma$ is the classical Gamma function; if $Q$ is an integer $n$, then $\omega_Q$ coincides with the Lebesgue measure of the unit Euclidean ball in $\setR^n$.   For a locally Ahlfors $Q$-regular metric measure space $(X,\dist,\mu)$,  we define
\[
\theta_r (x) \df \frac{\mu(B_r(x))}{\omega_Q r^Q}
\]
for any $x \in X$ and $r>0$, and
\[
\mu_r = \frac{1-\theta_r}{r} \, \mu .
\]
Each $\mu_r$ is a signed Radon measure. We let 
\[
|\mu_r| = \frac{|1-\theta_r|}{r} \, \mu
\]
be the associated total variation measure. Here is the main result of this paragraph.

\begin{prop}\label{prop:extension_of_AKS}
Let $(X,\dist,\mu)$ be a locally compact, locally Ahlfors $Q$-regular metric measure space satisfying
\begin{equation}\label{eq:zero_total_variation}
 |\mu_r| \weakto 0 \qquad \text{as $r \downarrow 0$.}
\end{equation}
Then $u \in \Lip(X)$ admits a weak $\AMV$ Laplacian if and only if $u$ admits a weak $\SAMV$ Laplacian, in which case $\bold{\Delta}^{\dist}_\mu u=\bold{\tilde{\Delta}}^{\dist}_\mu u$.
\end{prop}

To prove Proposition \ref{prop:extension_of_AKS}, we need the following lemma.

\begin{lemma}\label{lem:technical}
Let $(X,\dist,\mu)$ be a locally compact metric measure space satisfying a locally uniform comparability condition. For any compact set $K \subset X$,  there exist $r_K, C_K>0$ depending only on $K$, and another compact set $K' \subset X$ containing $K$ such that for any $f \in L^\infty_{\text{loc}}(X,\mu)$,
\begin{equation}\label{eq:A*-I}
\sup_{r \in (0,r_K)} \|A_r^*f - f \|_{L^\infty(K,\mu)} \le C_K \|f\|_{L^\infty(K',\mu)}.
\end{equation}
\end{lemma}

\begin{proof} Let $K \subset X$ be a compact set. The local compactness of $(X,\dist)$ ensures that there exist $r_K>0$ and another compact set $K' \subset X$ containing $K$ such that 
\[
\bigcup_{x \in K} B_{r_K}(x) \subset K'.
\]
Consider $f \in L^\infty_{\text{loc}}(X,\mu)$. For $\mu$-a.e.~$x\in K$ and any $r \in (0,r_K)$,
\begin{align*}
|A_r^* f(x) - f(x)| & = \left| \int_{B_r(x)} \left(\frac{f(y)}{\mu(B_r(y))}-\frac{f(x)}{\mu(B_r(x))}\right) \di \mu(y) \right|\\
& \le \left| \int_{B_r(x)} \frac{f(y)-f(x)}{\mu(B_r(y))} \di \mu(y) \right| + \left| \int_{B_r(x)} \left(\frac{f(x)}{\mu(B_r(y))}-\frac{f(x)}{\mu(B_r(x))}\right) \di \mu(y) \right|\\
& \le \int_{B_r(x)} \frac{|f(y)-f(x)|}{\mu(B_r(y))} \di \mu(y) + |f(x)| \left| \int_{B_r(x)}\frac{\di \mu(y)}{\mu(B_r(y))}-1\right|\\
& \le 2  \|f\|_{L^\infty(K',\mu)} \int_{B_r(x)} \frac{\di \mu(y)}{\mu(B_r(y))} + \|f\|_{L^\infty(K',\mu)} \left( \int_{B_r(x)} \frac{\di \mu(y)}{\mu(B_r(y))}+1\right)\\
& =  \|f\|_{L^\infty(K',\mu)}  \left( 3  \int_{B_r(x)} \frac{\di \mu(y)}{\mu(B_r(y))} + 1\right).
\end{align*}
By Remark \ref{rem:locally_compact}, the locally uniform comparability condition implies that
\[
\int_{B_r(x)} \frac{\di \mu(y)}{\mu(B_r(y))}  \le C'_K
\]
for some $C'_K>0$ depending only on $K$. We obtain the desired result by setting $C_K\df 3 C'_K +1$.

\end{proof}

We are now in a position to prove Proposition \ref{prop:extension_of_AKS}.

\begin{proof}[Proof of Proposition \ref{prop:extension_of_AKS}]
 By Example~\ref{ex:Ahlfors} our space fulfills a locally uniform comparability condition, and by Remark~\ref{rem:implies(1)} we infer that $z_r\in L^1_{\text{loc}}(X,\mu)$. Hence Corollary~\ref{cor:main} shows that we only need to prove that 
\[
\frac{z_r}{r} \mu \weakto 0
\]
as $r \downarrow 0$. Take $\phi \in C_c(X)$.  Up to decomposing $\phi$ as $\phi^+ - \phi^-$,  we do not lose any generality in assuming that $\phi$ is non-negative, what we do from now on.  Let $K$ be the support of $\phi$. Let $r_K, C_K$ and $K'$ be given by Lemma \ref{lem:technical}. Consider  $r \in (0,r_K)$.  Observe that if $x \in K$ and $y \in B_r(x)$, then
\begin{align*}
|\delta_r(x,y)| = \frac{|\mu(B_r(y))-\mu(B_r(x))|}{\mu(B_r(y))}  =  \frac{|\theta_r(y)-\theta_r(x)|}{\theta_r(y)} & \le \frac{1}{\theta_r(y)} (|\theta_r(y)-1| +  |1-\theta_r(x))|)\\
& \le \omega_Q^{-1} C_{K'} \bigg( |\theta_r(y) -1| +  |1-\theta_r(x)| \bigg)
\end{align*}
where we have multiplied and divided by $\omega_Q r^Q$ to get the second equality and we have used the local Ahlfors regularity property on $K'$ to get the last one. Then
\begin{align*}
\int_X \phi \, \frac{z_r}{r} \di \mu  & =   \int_X \phi(x) \fint_{B_r(x)} \frac{|\delta_r(x,y)|}{r} \di \mu(y) \di \mu(x)\\ & \le \omega_Q^{-1} C_{K'} \left( \int_X \phi(x) \fint_{B_r(x)} \frac{|\theta_r(y)-1|}{r} \di \mu(y) \di \mu(x) + \int_X \phi(x) \frac{|\theta_r(x)-1|}{r} \di \mu(x) \right)
\end{align*}
By \eqref{eq:zero_total_variation}, the second term on the right-hand side converges to $0$ when $r \downarrow 0$. Let us show that also the first one converges to 0.  We have
\begin{align*}
\int_X \phi(x) \fint_{B_r(x)} \frac{|\theta_r(y)-1|}{r} \di \mu(y) \di \mu(x) & = \int_X \phi(x) A_r \left[ \frac{|\theta_r(\cdot)-1|}{r} \right](x) \di \mu(x)\\
& = \int_X A_r^*\phi(x) \frac{|\theta_r(x)-1|}{r} \di \mu(x).
\end{align*}
Since $x \notin K'$ yields $B_r(x) \cap \supp \phi = \emptyset$ which in turn implies that $A_r^*\phi(x)=0$, we may replace the previous integral over $X$ with an integral over $K'$. Then
\begin{align*}
\int_X \phi(x) \fint_{B_r(x)} \frac{|\theta_r(y)-1|}{r} \di \mu(y) \di \mu(x)& \le \int_{K'} |A_r^*\phi(x)- \phi(x)| \frac{|\theta_r(x)-1|}{r} \di \mu(x) + \int_{K'} \phi \frac{|\theta_r-1|}{r} \di \mu \\
& \le C \|\phi\|_{L^\infty(X,\mu)}\int_{K'} \frac{|\theta_r-1|}{r} \di \mu\\
& \le C \|\phi\|_{L^\infty(X,\mu)}\int_{X} \psi \frac{|\theta_r-1|}{r} \di \mu
\end{align*}
where we have used \eqref{eq:A*-I} to get the second inequality and we have set $\psi(\cdot)\coloneqq (1-\dist(\cdot,K')/\rho)^{+}$ for $\rho>0$ small enough to ensure that $\psi \in C_c(X)$. By \eqref{eq:zero_total_variation},  we get
\[
\lim\limits_{ r\downarrow 0} \int_{X} \psi \frac{|\theta_r-1|}{r} \di \mu = 0.
\]

\end{proof}

\begin{rem}
In \cite{KLP}, the authors introduced the following definition: a metric measure space $(X,\dist,\mu)$ has \textit{vanishing metric measure boundary} if $\mu_r \weakto 0$ as $r \downarrow 0$. It would be interesting to study whether a result like Proposition \ref{prop:extension_of_AKS} may be obtained with \eqref{eq:zero_total_variation} replaced by this weaker assumption.
\end{rem}

\begin{rem}\label{rem:extension_of_AKS}
For $K \in \setR$ and $N \in [1,+\infty)$, an $\RCD(K,N)$ space is a proper (hence locally compact) metric measure space satisfying a synthetic notion of 
Ricci curvature bounded below by $K$ and dimension bounded above by $N$; we refer to \cite{AmbrosioICM,Gigli}, for instance, for a nice account on these spaces.  According to \cite{DePG}, an $\RCD(K,N)$ space $(X,\dist,\mu)$ is called non-collapsed if $N$ is an integer and $\mu = \haus^N$.  Non-collapsed $\RCD(K,N)$ spaces are locally Ahlfors $N$-regular; this is a consequence of \cite[Theorem 1.3]{DePG}. By \cite[Theorem 1.2]{BMS22} (see also \cite[Theorem 3.7]{AKS_2}),  any non-collapsed $\RCD(K,N)$ space with vanishing metric-measure boundary satisfies \eqref{eq:zero_total_variation}. Building upon this,  Adamowicz, Kijowski and Soultanis proved in \cite[Corollary 3.9]{AKS_2} that the weak $\AMV$ and $\SAMV$ Laplacians coincide on such a space. Our result provides the same conclusion in a setting where no curvature-dimension condition is assumed. In this regard, it would be worth investigating the validity of \eqref{eq:zero_total_variation} in contexts which are not $\RCD$, like sub-Riemannian or Finsler spaces.
\end{rem}

\section{Weighted Lebesgue measures}

In this section, we study the context of weighted Lebesgue measures where the two notions of pointwise Laplacian \emph{do not} coincide. We first focus on points where the weight is positive, before tackling points where the weight vanishes.

\subsection{Positive weights}

We begin this subsection with a straightforward result.

\begin{prop}\label{prop:positive_weights}
Let $\Omega \subset \setR^n$ be an open set equipped with the Euclidean distance $\dist_e$, and $\mu\df w \leb^n \measrestr \Omega$ where $w \in C^1(\Omega)$. Consider $(\Omega, \dist_e,\mu)$. Then for any $x \in \Omega$ such that $w(x)>0$, any $u  \in C^2(\Omega)$ admits a pointwise $\AMV$ Laplacian and a pointwise $\SAMV$ Laplacian at $x$, and
\begin{equation}\label{eq:first_equality}
\Delta_{\mu}^{\dist_{e}}u(x) = \frac{1}{2(n+2)}(\Delta u(x) + 2\langle\nabla \ln w,\nabla u\rangle(x)),
\end{equation}
\begin{equation}\label{eq:second_equality}
\tilde{\Delta}_{\mu}^{\dist_{e}}u(x) = \frac{1}{2(n+2)}(\Delta u(x) + \langle\nabla \ln w,\nabla u\rangle(x)) \, \cdot
\end{equation}
\end{prop}

\begin{proof}
The first equality was proved in \cite[Proposition 2.3]{MT}.  Let us prove the second equality following the same lines. Let $r_0>0$ be such that $B_{r_0}(x) \subset \Omega$ and $w(y) > w(x)/2$ for any $y \in B_{r_0}(x)$. Then for any $r \in (0,r_0)$,
\begin{align}\label{eq:pirate}
\tilde{\Delta}_{\mu,r}^{\dist_{e}}u(x) & = \frac{1}{2}\left(\Delta_{\mu,r}^{\dist_{e}}u(x)+I(r)\right).
\end{align}
where
\[
I(r)\df \frac{1}{r^{2}}\int_{B_{r}(x)}(u(y)-u(x))\frac{w(y)}{\mu(B_{r}(y))}\di y.
\]
We claim that
\begin{equation}\label{eq:claim_I(r)}
I(r) \to \frac{\Delta u(x)}{2(n+2)}\qquad \text{as $r\downarrow 0$}.
\end{equation}
Once this claim is proved,  \eqref{eq:second_equality} follows from letting $r \downarrow 0$ in \eqref{eq:pirate} and using \eqref{eq:first_equality}.  We now prove \eqref{eq:claim_I(r)}. Since $w$ is $C^1$, a first-order Taylor expansion shows that for any $0<r<r_0/2$, $y \in B_r(x)$, and $z \in B_r(y)$, 
\[
w(z) = w(y) + \langle R_y(z),y-z \rangle
\]
for some $R_y(z) \in \setR^n$ such that
\[
|R_y(z)| \le C \coloneqq \sup_{\xi \in \overline{B}_{r_0}(x)} |\nabla w (\xi)|.
\]
Then
\[\mu(B_{r}(y))=\int_{B_{r}(y)}w(z)\di z = \leb^{n}(B_{r}(y))\left( w(y)+ \fint_{B_r(y)}\langle R_y(z),y-z \rangle \di z \right)\]
so that 
\begin{equation}\label{eq:quotient}
\frac{w(y)}{\mu(B_{r}(y))} = \frac{1}{\leb^{n}(B_{r}(y))\left( 1+ \frac{1}{w(y)}\fint_{B_r(y)}\langle R_y(z),y-z \rangle \di z \right)}\, \cdot
\end{equation}
Now
\begin{equation}\label{eq:O(r)}
\left| \frac{1}{w(y)}\fint_{B_r(y)}\langle R_y(z),y-z \rangle \di z \right| \le \frac{2C }{w(x)} \, r = O(r)
\end{equation}
and $\leb^{n}(B_{r}(y))=\leb^{n}(B_{r}(x))$, hence we get
\[
I(r) = \frac{1}{r^{2}(1+O(r))}\fint_{B_{r}(x)}u(y)-u(x)\,\di y\to\frac{\Delta u(x)}{2(n+2)}\qquad \text{as $r\downarrow 0$}.\]
\end{proof}

\begin{rem}
Proposition \ref{prop:positive_weights} shows that if $w>0$ everywhere in $\Omega$, then $\tilde{\Delta}_{\mu}^{\dist_{e}}$ coincides with a dimensional constant times the classical weighted Laplacian of $(\setR^n,\dist_e,w\leb^n)$ (also known as drifted Laplacian, or $f$-Laplacian, or Witten Laplacian,  see e.g.~\cite{munteanu2011smooth,colbois2015eigenvalues} and the references therein), namely
\[
\Delta u + \langle\nabla \ln w,\nabla u\rangle.
\]
Moreover, working in exponential coordinates, it is not difficult to show that \eqref{eq:first_equality} and \eqref{eq:second_equality} extend to the setting of weighted Riemannian manifolds $(M,g,w\mu)$ where $\mu$ is the canonical Riemannian measure and $w : M \to (0,+\infty)$ is a smooth map.
\end{rem}

We now study the case where the Euclidean distance on $\Omega$ is replaced with a more general distance $\dist$. To this aim,  for any $x \in \Omega$ and $r>0$, we set
\[
M^{r}(x) \df \left( M_{ij}^{r}(x) \df \frac{1}{r^2} \fint_{B_r(x)}  (y-x)_i (y-x)_j \di y \right)_{1 \le i<j\le n} .
\]
In case $\dist$ is associated with a norm, the change of variable $\xi=(y-x)/r$ shows that the matrices $M^{r}(x)$ are all equal to the second-moment matrix
\[
M(0)  \df \left( M_{ij}(0) \df  \fint_{B_1(0)}  \xi_i \xi_j \di \xi \right)_{1 \le i<j\le n}.
\]
We will also work under the assumption that balls for $\dist$ are symmetric with respect to the vector space structure of $\setR^n$, in the sense that for any $x \in \Omega$ and $r>0$ such that $B_r(x) \subset \Omega$, for any $v \in \setR^n$,
\begin{equation}\label{eq:symmetricball}
x+v \in B_r(x) \iff x-v \in B_r(x).
\end{equation}
This is trivially satisfied when $\dist$ is associated with a norm. But it may fail in general.  For instance, set
\[
\dist(x,y)\df\begin{cases}
\lvert x-y\rvert, & x,y\le0,\\
\frac{1}{2}|x-y|, & x,y>0,\\
\frac{1}{2}y-x, & x\le0,\,y>0,\\
\frac{1}{2}x-y, & y\le0,\,x>0.
\end{cases}
\]
Then $\dist$ is a metric on $\setR$ such that  $\frac{1}{2}\dist_{\text{e}}\le \dist\le2\dist_{\text{e}}$ and $B_{r}(0)=(-r,2r)$ for any $r>0$. 

A large class of distances $\dist$ which satisfy \eqref{eq:symmetricball} without being associated with a norm is given by the following: For $\alpha=(\alpha_1,\ldots,\alpha_n) \in \{1,2\}^n$, set
\[
\Phi_\alpha (x) \df (x_1^{\alpha_1},\ldots,x_n^{\alpha_n})
\]
for any $x=(x_1,\ldots,x_n)\in (0,+\infty)^n$. Set
\[
\dist_\alpha (x,y) \df \|\Phi_\alpha(x)-\Phi_\alpha(y)\|
\]
for any $x,y \in (0,+\infty)^n$, where $\|\cdot\|$ is any $l^p$-norm on $\setR^n$. Then $\dist_\alpha$ is a distance on $(0,+\infty)^n$ which always satisfies \eqref{eq:symmetricball}; it is associated with a norm if and only if $\alpha_1 = \ldots = \alpha_n=1$, because otherwise homogeneity fails. 

Our next result makes use of $z^{\leb^{n}}_r$, i.e.~the average of the absolute value of the distortion function, defined for any $r>0$ by
\begin{equation}\label{eq:Lebesgue_distortion}
z^{\leb^n}_r(x) \coloneqq \fint_{B_r(x)}\left|1-\frac{\leb^n(B_{r}(x))}{\leb^n(B_{r}(y))}\right| \di y.
\end{equation}
for any $x \in \setR^n$.

\begin{prop}\label{prop:sym_Hausdorff_distortion}
Let $\Omega \subset \setR^n$ be open with respect to a distance $\dist$ satisfying \eqref{eq:symmetricball}.  For $w \in C^1(\Omega)$ set $\mu\df w \leb^n \measrestr \Omega$. Consider $(\Omega, \dist,\mu)$.  Let $x \in \Omega$ be such that $w(x)>0$. 
\begin{enumerate}
\item Assume that the limit $M(x)\df \lim_{r \downarrow 0^+} M^r(x)$ exists.  Then any $u  \in C^2(\Omega)$ admits a pointwise $\AMV$ Laplacian at $x$, and
\begin{equation}\label{eq:first_equality_2}
\Delta_{\mu}^{\dist}u(x)	=\frac{1}{2}\Tr(M(x)\nabla^2 u(x))+\frac{1}{w(x)}\langle\nabla w(x),M(x)\nabla u(x)\rangle.
\end{equation}
\item If in addition
\begin{equation}\label{eq:distortion_r}
z^{\leb^{n}}_r(x) =o(r)
\end{equation}
as $r \downarrow 0$, then $u$ admits a pointwise $\SAMV$ Laplacian at $x$, and
\begin{equation}\label{eq:second_equality_2}
\tilde{\Delta}_{\mu}^{\dist}u(x)	=\frac{1}{2w(x)}\Tr(M(x)\nabla(w\nabla u)(x)).
\end{equation}
\end{enumerate}
\end{prop}

\begin{proof} 
\textbf{Step 1.} We first prove \eqref{eq:first_equality_2}.  To this aim, we claim that for any $r >0$,
\begin{equation}\label{eq:claim_r}
\Delta_{\mu,r}^{\dist}u(x)= \frac{\leb^{n}(B_{r}(x))w(x)}{\mu(B_{r}(x))} \left( \frac{1}{2}\Tr\left(\nabla^{2}u(x)M^{r}(x)\right) + \frac{1}{w(x)} \langle M^{r}(x)\nabla u(x),\nabla w(x)\rangle \right) + O(r).
\end{equation}
Once this is proved, \eqref{eq:first_equality_2} follows immediately since $\mu(B_{r}(x))/\leb^{n}(B_{r}(x)) \to w(x)$ as $r \downarrow 0$.  Let us prove \eqref{eq:claim_r}. By a second-order Taylor expansion of $u$ and a first-order Taylor expansion of $w$,  for any $y \in B_r(x)$ we infer that
\[
(u(y)-u(x))w(y) = \sum_{i=1}^n [wu_i](x) (y-x)_i +  \sum_{i,j=1}^{n}\left[\frac{wu_{ij}}{2} + w_i u_j\right](x) \, (y-x)_i (y-x)_j + O(r^3)
\]
where $O(r^3)$ is independent of $y$ and where we have denoted by $u_i$ the partial derivatives of $u$, by $u_{ij}$ the second order partial derivatives of $u$, and similarly for $w$.  Integrate with respect to $y \in B_{r}(x)$ and then divide by $r^2 \mu(B_r(x))$.  Then the left-hand side becomes $\Delta_{\mu,r}^{\dist}u(x)$. Due to antisymmetry, Assumption \eqref{eq:symmetricball} implies that the first term on the right-hand side vanishes when integrating w.r.t.~$y$ over $B_{r}(x)$.  Then we get
\begin{align*}
\Delta_{\mu,r}^{\dist}u(x) & =  \sum_{i,j=1}^{n} \left[\frac{wu_{ij}}{2} + w_i u_j\right](x) \frac{1}{r^{2}\mu(B_{r}(x))} \int_{B_r(x)}\ \, (y-x)_i (y-x)_j  \di y + \frac{\leb^n(B_r(x))}{\mu(B_{r}(x))} O(r) \\
& = \frac{\leb^n(B_r(x))}{\mu(B_{r}(x))}   \sum_{i,j=1}^{n} \left[\frac{wu_{ij}}{2} + w_i u_j\right](x) M^r_{ij}(x) + O(r)\\
& =  \frac{\leb^{n}(B_{r}(x))}{\mu(B_{r}(x))} \left( \frac{w(x)}{2}\Tr\left(\nabla^{2}u(x)M^{r}(x)\right) + \langle M^{r}(x)\nabla u(x),\nabla w(x)\rangle \right) + O(r).
\end{align*}

\textbf{Step 2.} Now we prove \eqref{eq:second_equality_2}. Observe that
\begin{equation}\label{eq:pirate_2}
\tilde{\Delta}^\dist_{\mu,r}u(x)= \frac{1}{2}\left( \Delta^\dist_{\mu,r}u(x) +I(r) \right)
\end{equation}
where
\[
I(r)\df \frac{1}{r^{2}}\int_{B_{r}(x)}(u(y)-u(x))\frac{w(y)}{\mu(B_{r}(y))}  \di y.
\]
We claim that
\begin{equation}\label{eq:Tr}
I(r) \to  \Delta_{\leb^n}^\dist u(x) = \frac{1}{2}\Tr (M(x) \nabla^2u(x)).
\end{equation}
Once this claim is proved, letting $r \downarrow 0$ in \eqref{eq:pirate_2} yields \eqref{eq:second_equality_2}, since a direct computation gives
	\begin{equation}
\frac{1}{2w(x)}\Tr(M(x)\nabla(w\nabla u)(x))=	\frac{1}{2}\Tr(M(x)\nabla^{2}u(x))+\frac{1}{2w(x)}\langle M(x)\nabla u(x),\nabla w(x)\rangle.
	\end{equation}
	Thus we prove \eqref{eq:Tr}. Thanks to \eqref{eq:quotient} and \eqref{eq:O(r)}, we get
	\begin{align}\label{eq:I(r)}
	I(r) & =\frac{1}{r^{2}(1+O(r))}\int_{B_{r}(x)}(u(y)-u(x))\frac{1}{\leb^{n}(B_{r}(y))} \di y \nonumber \\
	& = \frac{1}{r^{2}(1+O(r))}\Bigg[\fint_{B_{r}(x)}(u(y)-u(x)) \di y \nonumber \\
	& \qquad-\fint_{B_{r}(x)}(u(y)-u(x))\left(1-\frac{\leb^{n}(B_{r}(x))}{\leb^{n}(B_{r}(y))}\right) \di y\Bigg] \nonumber\\
	& = \frac{1}{1+O(r)}\Bigg[\Delta_{\leb^n,r}^\dist u(x)-\frac{1}{r^2}\fint_{B_{r}(x)}(u(y)-u(x))\left(1-\frac{\leb^{n}(B_{r}(x))}{\leb^{n}(B_{r}(y))}\right) \di y\Bigg].
\end{align}	
Now using that $u$ is Lipschitz, we get
\begin{align*}
\frac{1}{r^2} \left| \fint_{B_{r}(x)}(u(y)-u(x))\left(1-\frac{\leb^{n}(B_{r}(x))}{\leb^{n}(B_{r}(y))}\right) \di y \right|  & \le \frac{\Lip(u)}{r^2}   \fint_{B_r(x)} \dist(x,y)\left|1-\frac{\leb^{n}(B_{r}(x))}{\leb^{n}(B_{r}(y))}\right|  \di y \\
& \le  \frac{\Lip(u)}{r}   \fint_{B_r(x)} \left|1-\frac{\leb^{n}(B_{r}(x))}{\leb^{n}(B_{r}(y))} \right| \di y\\
& =  \Lip(u) \frac{z^{\leb^{n}}_r(x)}{r}  \to 0 \qquad \text{as $r \downarrow 0$}
\end{align*}
by Assumption $\eqref{eq:distortion_r}$.	Then \eqref{eq:Tr} follows from letting $r\downarrow 0$ in \eqref{eq:I(r)}.
\end{proof}

\begin{rem}
It follows from the previous proof that 
\begin{equation}\label{eq:decompo}
\tilde{\Delta}_{\mu}^{\dist}u(x)  = \frac{1}{2} \left(\Delta_{\mu}^{\dist}u(x) + \Delta_{\leb^n}^{\dist}u(x)  \right).
\end{equation}
\end{rem}

\begin{rem}
Let $\Vert \cdot \Vert$ denote the infinity matrix norm. By the Jensen inequality,
\[
\left\| M^{r}(x)  \right\| \le  \fint_{B_r(x)} \left\| \left( \frac{(y-x)_i(y-x)_j}{r^2}\right)_{1 \le i \le j \le n} \right\| \di y \le 1
\]
for any $r>0$, so the sequence $\{M^{r}(x)\}_{r>0}$ is bounded and therefore it admits a set of accumulation points $\{M_\alpha\}$ as $r \downarrow 0$. This set is a singleton if and only if the limit $M(x)\df \lim_{r \downarrow 0} M^r(x)$ exists. If this is not the case, then $\Delta_{\mu}^{\dist}u(x)$ and $\tilde{\Delta}_{\mu}^{\dist}u(x)$ may be understood as multivalued:
\begin{align*}
\Delta_{\mu}^{\dist}u(x)	& =\left\{\frac{1}{2}\Tr(M_\alpha(x)\nabla^2 u(x))+\frac{1}{w(x)}\langle\nabla w(x),M_\alpha(x)\nabla u(x)\rangle\right\}, \\
\tilde{\Delta}_{\mu}^{\dist}u(x)	& =\left\{\frac{1}{2w(x)}\Tr(M_\alpha(x)\nabla(w\nabla u)(x))\right\}.
\end{align*}
\end{rem}

\subsection{Vanishing weights}

We are now interested in the pointwise $\AMV$ and $\SAMV$ Laplacians  at a point where the weight vanishes. With no loss of generality, we assume that this point is the origin $0_n$.  We let $\|\cdot\|$ be a norm on $\setR^n$.

\begin{D}
Let $w$ be a Lebesgue integrable function defined on a neighborhood $\Omega$ of $0_n$. We say that $w$ is a vanishing weight at $0_n$ if $w(x) > 0$ for $\leb^n$-a.e.~$x \in \Omega$ and $0_n \in \Leb(w)$ with $w^*(0_n)=0$.
\end{D}

For any Radon measure $\nu$ on $\setR^n$, we define the second-moment matrix of $\nu$ as
\begin{equation}\label{eq:second_moment_matrix}
M_\nu \df \left( \int_{B_1(0_n)} y_i y_j \di \nu(y) \right)_{1 \le i,j \le n}.
\end{equation}

\subsubsection{Infinitesimally even weights }

\begin{D}
We say that a weight $w$ is infinitesimally even at $0_n$ if it is $\leb^n$-essentially bounded in a neighborhood of $0_n$ and such that
\[
 \text{$\leb^n$-}\esssup\limits_{x \in B_r(0_n)} |w(x)-w(-x)|=o\left( r \fint_{B_r(0_n)} w \di \leb^n \right) \qquad \text{as $r \downarrow 0$.}
\]
\end{D}

Our first result is the following.

\begin{prop}\label{prop:as_even_weights}  Consider $(\Omega,\|\cdot\|,w \leb^n)$ where $w$ is a vanishing weight at $0_n$ of domain $\Omega$ which is infinitesimally even at $0_n$. Set $\mu \df w \leb^n$. 
Assume that the blow-up probability measures 
\[  \nu_r \df \left( \frac{w(r\, \cdot)r^n}{\mu(B_r(0_n))} \right) \leb^n \measrestr B_1(0_n)
\]
weakly converge to some Radon measure $\nu$ supported in $B_1(0_n)$.\footnote{observe that
\[
\frac{1}{r^2} \fint_{B_r(0_n)} x_ix_j \di \mu(x) = \int_{B_1(0_n)} y_i y_j \di \nu_r(y) 
\]}
Then the pointwise $\AMV$ Laplacian at $0_n$ of any function $u : \Omega \to \setR$ two times differentiable at $0_n$ exists and satisfies
\begin{equation}\label{eq:AMV_in_0}
\Delta^{\|\cdot\|}_{\mu}u(0_n)=\frac{1}{2} \Tr (M_\nu \nabla^2u)(0_n).
\end{equation}
\end{prop}

\begin{proof}
From the Taylor theorem,  there exist a neighborhood $\Omega' \subset \Omega$ of $0_n$ and functions $h_{ij}: \Omega' \to \setR$ such that  $\lim\limits_{x \to 0_n} h_{ij}(x)=0$ and
\[
u(x) - u(0_n) = \sum_{i=1}^n \frac{\partial u}{ \partial x_i}(0_n) x_i +  \frac{1}{2}\sum_{i,j=1}^n \frac{\partial^2 u}{ \partial x_i \partial x_j}(0_n)x_ix_j +  \sum_{i,j=1}^n h_{ij}(x)x_ix_j
\] 
for any $x \in \Omega'$. Then for any $r>0$ such that $B_r(0_n) \subset \Omega'$,
\begin{align}\label{eq:step0}
\Delta^{\|\cdot\|}_{\mu,r}u(0_n) =  \sum_{i=1}^n \frac{\partial u}{ \partial x_i}(0_n) \frac{1}{r^2} \fint_{B_r(0_n)} x_i \di \mu(x) & +  \frac{1}{2} \sum_{i,j=1}^n \frac{\partial^2 u}{ \partial x_i \partial x_j}(0_n)\frac{1}{r^2} \fint_{B_r(0_n)} x_i x_j\di \mu(x)\nonumber \\ &  +  \sum_{i,j=1}^n \frac{1}{r^2} \fint_{B_r(0_n)} h_{ij}(x)x_ix_j \di \mu(x). 
\end{align}

\hfill

\noindent \textbf{Step 1.} We fix $i,j \in \{1,\ldots,n\}$ and show that
\begin{equation}\label{eq:step2}
\lim\limits_{r \downarrow 0} \frac{1}{r^2} \fint_{B_r(0_n)} h_{ij}(x)x_ix_j \di \mu(x)  = 0.
\end{equation}
Fix $\eps>0$. Then there exists $\delta \in (0,r_0)$ such that $|h_{ij}(x)|<\eps$ for any $x \in B_\delta(0_n)$. For any $r \in (0,\delta)$,
\[
\left| \frac{1}{r^2} \fint_{B_r(0_n)} h_{ij}(x)x_ix_j \di \mu(x) \right| \le \frac{1}{r^2 } \fint_{B_r(0_n)} |h_{ij}(x)| |x_i| |x_j| \di \mu(x) < \eps,
\]
where we have used that $|x_i| \le r$ and $|x_j|\le r$ to get the last inequality. This yields \eqref{eq:step2}.\\

\noindent \textbf{Step 2.} We fix $i \in \{1,\ldots,n\}$ and show that
\begin{equation}\label{eq:step1}
\lim\limits_{r \downarrow 0} \frac{1}{r^2}\fint_{B_r(0_n)} x_i \di \mu(x)  = 0.
\end{equation}
Observe that
\begin{equation}\label{eq:change}
\fint_{B_r(0_n)} x_i \di \mu(x) = \frac{1}{\int_{B_r(0_n)} w(x)\di x}\int_{B_r(0_n)} x_i w(x) \di x
\end{equation}
and
\begin{align}\label{eq:change1/2}
\int_{B_r(0_n)} x_i w(x) \di x & = \frac{1}{2} \int_{B_r(0_n)} x_i w(x) \di x + \frac{1}{2} \int_{B_r(0_n)} x_i w(x) \di x  \nonumber \\
& = \frac{1}{2} \int_{B_r(0_n)} x_i w(x) \di x + \frac{1}{2} \int_{B_r(0_n)} -x_i' w(-x') \di x' \nonumber \\
& = \frac{1}{2}\int_{B_r(0_n)} x_i [w(x)-w(-x)] \di x,
\end{align}
where we have used the change of variable $x'=-x$ in the second integral to pass from the first line to the second.  Set $S_w(r)\coloneqq \text{$\leb^n$-}\esssup\limits_{x \in B_r(0_n)} |w(x)-w(-x)|.$ Then
\begin{align*}
\left|\frac{1}{r^2}\fint_{B_r(0_n)} x_i \di \mu(x)\right| & = \frac{1}{2r^2|\int_{B_r(0_n)} w(x)\di x|} \left| \int_{B_r(0_n)} x_i [w(x)-w(-x)] \di x \right| \\
& \le \frac{1}{2r^2|\int_{B_r(0_n)} w(x)\di x|} \left( \int_{B_r(0_n)} |x_i| |w(x)-w(-x)| \di x \right) \\
& \le \frac{S_w(r)}{2r|\int_{B_r(0_n)} w(y)\di y|}  \int_{B_r(0_n)} 1 \di x \\
& = \frac{S_w(r)}{2r|\fint_{B_r(0_n)} w(y)\di y|}
\end{align*}
where we have used $|x_i|\le r$ and $|w(x)-w(-x)|\le S_w(r)$ to get the second inequality. Then \eqref{eq:step1} follows from the infinitesimal evenness assumption on $w$.\\

\noindent \textbf{Step 3.} We fix $i,j \in \{1,\ldots,n\}$ and show that
\begin{equation}\label{eq:step3}
\lim\limits_{r \downarrow 0} \frac{1}{r^2} \fint_{B_r(0_n)} x_ix_j \di \mu(x)  = \int_{B_1(0_n)} y_iy_j \di \nu(y).
\end{equation}
Performing the change of variable $y=x/r$,  we have
\begin{align*}
\frac{1}{r^2} \fint_{B_r(0_n)} x_ix_j \di \mu(x) & = \frac{1}{r^2 \mu(B_r(0_n))} \int_{B_1(0_n)} (ry_i)(ry_j)w(ry) r^n  \di y\\
& = \int_{B_1(0_n)} y_i y_j \di \nu_r(y)  \to\int_{B_1(0_n)} y_i y_j \di \nu(y) \qquad \text{as }r\downarrow 0.
 \end{align*}
\end{proof}

\begin{rem}\label{rem:Prok}
By the Prokhorov theorem, the probability measures $\{\nu_r\}$ always admit accumulation points in the weak topology of measures as $r \downarrow 0$.   The assumption made in the previous proposition demands that there is a unique limit. If this is not satisfied, then $\Delta^{\|\cdot\|}_{\mu}u(0_n)$ may be understood as multivalued, namely we get
\[
\Delta^{\|\cdot\|}_{\mu}u(0_n)=\left\{ \frac{1}{2} \Tr (M_{\bar{\nu}_\alpha} \nabla^2u)(0_n)\right\}
\]
where $\{\bar{\nu}_\alpha\}$ are the limit points of $\{\nu_r\}$ as $r \downarrow 0$.
\end{rem}

A version of Proposition \ref{prop:as_even_weights} also holds in the case of the $\SAMV$ Laplacian, but under slightly different assumptions.

\begin{prop}\label{prop:sym_as_even_weights}  Consider $(\Omega,\|\cdot\|,w \leb^n)$ where $w$ is a vanishing weight at $0_n$ of domain $\Omega$ which is infinitesimally even at $0_n$.  Assume that $\mu \df w \leb^n$  satisfies a comparability condition at $0_n$, that the weight
\begin{equation}\label{eq:tildew}
\tilde{w}(\cdot) \coloneqq \frac{w(\cdot)}{2}\left( 1 + \frac{\mu(B_r(0_n))}{\mu(B_r(\cdot))} \right)
\end{equation}
is infinitesimally even at $0_n$,  and that there exists a Radon measure $\tilde{\nu}$ on $B_1(0_n)$ such that
\begin{equation}\label{eq:prob_meas_2}
\left( \frac{\tilde{w}(r\, \cdot)r^n}{\mu(B_r(0_n))} \right) \leb^n \measrestr B_1(0_n) \fd \tilde{\nu}_r  \weakto  \tilde{\nu} \qquad \text{as $r \downarrow 0$.}
\end{equation}
Then the pointwise $\SAMV$ Laplacian at $0_n$ of any function $u : \Omega \to \setR$ two times differentiable at $0_n$ exists and satisfies
\begin{equation}\label{eq:SAMV_in_0}
\tilde{\Delta}^{\|\cdot\|}_{\mu}u(0_n)=\frac{1}{2} \Tr (M_{\tilde{\nu}} \nabla^2u)(0_n).
\end{equation}
\end{prop}

\begin{proof}
Let $u : \Omega \to \setR$ be two times differentiable at $0_n$.  Using the Taylor theorem like in the proof of Proposition \ref{prop:as_even_weights},  we get that for any small enough $r$, it holds that
\begin{align}\label{eq:sym_step0}
\tilde{\Delta}_{\mu,r}^{\lVert\cdot\rVert}u(0_n) & =  \sum_{i=1}^n \frac{\partial u}{ \partial x_i}(0_n) \frac{1}{r^2 \mu(B_r(0_n))} \int_{B_r(0_n)} x_i \tilde{w}(x)\di x \nonumber \\
& +  \sum_{i,j=1}^n \frac{\partial^2 u}{ \partial x_i \partial x_j}(0_n)\frac{1}{r^2\mu(B_r(0_n))} \int_{B_r(0_n)} x_i x_j \tilde{w}(x) \di x\nonumber \\ &  +  \sum_{i,j=1}^n \frac{1}{r^2\mu(B_r(0_n))} \int_{B_r(0_n)} h_{ij}(x)x_ix_j \tilde{w}(x) \di x.
\end{align}

Fix $i,j \in \{1,\ldots,n\}$.  For any $\eps>0$ there exists $\delta \in (0,r_0)$ such that $|h_{ij}(x)|<\eps$ for any $x \in B_\delta(0_n)$. For any $r \in (0,\delta)$,
\begin{align*}
\left| \frac{1}{r^2\mu(B_r(0_n))} \int_{B_r(0_n)} h_{ij}(x)x_ix_j \tilde{w}(x) \di x \right| & \le \frac{1}{r^2 \mu(B_r(0_n))} \int_{B_r(0_n)} |h_{ij}(x)| |x_i| |x_j| \tilde{w}(x) \di x\\
& \le \frac{\eps}{\mu(B_r(0_n))} \int_{B_r(0_n)} \tilde{w}(x) \di x  \le (1+C) \eps/2,
\end{align*}
where we have used the local comparability condition at $0_n$ to get the last inequality. This yields 
\begin{equation}\label{eq:sym_step2}
\lim\limits_{r \downarrow 0} \frac{1}{r^2\mu(B_r(0_n))} \int_{B_r(0_n)} h_{ij}(x)x_ix_j \tilde{w}(x) \di x = 0.
\end{equation}

Moreover, we have that
\begin{equation}\label{eq:sym_step1}
\lim\limits_{r \downarrow 0}\frac{1}{r^2\mu(B_r(0_n))} \int_{B_r(0_n)}x_i  \tilde{w}(x) \di x   = 0.
\end{equation}
Indeed, acting like in \eqref{eq:change1/2}
 we get
\[
\int_{B_r(0_n)} x_i  \tilde{w}(x) \di x = \frac{1}{2}\int_{B_r(0_n)} x_i [ \tilde{w}(x)- \tilde{w}(-x)] \di x.
\]
Set $S_ {\tilde{w}}(r)\coloneqq  \text{$\leb^n$-}\esssup\limits_{x \in B_r(0_n)} | \tilde{w}(x)- \tilde{w}(-x)|.$ Like in the previous proof, we obtain
\begin{align*}
\left|\frac{1}{r^2\mu(B_r(0_n))}\int_{B_r(0_n)} x_i \tilde{w}(x)\di x\right| &  \le  \frac{S_{ \tilde{w}}(r)}{2r|\fint_{B_r(0_n)} w(x)\di x|} \le \frac{(1+C)S_{ \tilde{w}}(r)}{4r|\fint_{B_r(0_n)} \tilde{w}(x)\di x|}
\end{align*}
where we have used the local comparability condition  at $0_n$ to get the last inequality. Then \eqref{eq:sym_step1} follows from the infinitesimal evenness assumption on $\tilde{w}$.

Lastly, acting as in the previous proof and using the definition of $\tilde{\nu}_r$,  we obtain that
\begin{equation}\label{eq:sym_step3}
\lim\limits_{r \downarrow 0} \frac{1}{r^2\mu(B_r(0_n))} \int_{B_r(0_n)} x_ix_j \tilde{w}(x)\di x  = \int_{B_1(0_n)} y_iy_j \di \tilde{\nu}(y)
\end{equation}
for any $i,j \in \{1,\ldots,n\}$.  

 The conclusion follows by letting $r \downarrow 0$ in \eqref{eq:sym_step0} and using \eqref{eq:sym_step2}, \eqref{eq:sym_step1} and \eqref{eq:sym_step3}.
\end{proof}

\begin{rem}
A simple computation shows that
\begin{equation}\label{eq:tilde}
\tilde{\nu}_r = \frac{1}{2} \left(\nu_r + \frac{r^n}{\mu(B_r(\cdot))}\mu \measrestr B_1(0_n)\right).
\end{equation}
In this way, it clearly appears that there exists $r_0>0$ such that the measures $\{\tilde{\nu}_r\}_{0<r<r_0}$ are uniformly bounded,  thanks to the local comparability condition at $0_n$ of $\mu$. Therefore, like in the non-symmetrized case (see Remark \ref{rem:Prok}),  the Prokhorov theorem implies that the measures $\{\tilde{\nu}_r\}_{r>0}$ admit limit points in the weak topology of measures as $r \downarrow 0$.  If these limit points are not unique then $\tilde{\Delta}^{\|\cdot\|}_{\mu}u(0)$ may be understood as multivalued, namely
\[
\tilde{\Delta}^{\|\cdot\|}_{\mu}u(0_n)=\left\{ \frac{1}{2} \Tr (M_{\hat{\nu}_\alpha} \nabla^2u)(0_n)\right\}
\]
where $\{\hat{\nu}_\alpha\}$ are the limit points of $\{\tilde{\nu}_r\}$ as $r \downarrow 0$.
\end{rem}

\begin{rem}\label{rk:evenness}
If $w$ is even,  then so is $\tilde{w}$. Indeed, in this case, obvious changes of variable show that
\[
\mu(B_r(-x)) = \int_{B_r(0_n)}w(-x+y)\di y = \int_{B_r(0_n)}w(-x-y)\di y = \int_{B_r(0_n)}w(x+y)\di y = \mu(B_r(x))
\]
from which evenness of $\tilde{w}$ follows from the definition in \eqref{eq:tildew}.
\end{rem}

\subsubsection{Weights $w(\cdot)=|\cdot|^{\alpha}$}
Proposition \ref{prop:as_even_weights} may be applied to the case of weights $w(\cdot)=|\cdot|^{\alpha}$ where $|\cdot|$ is the Euclidean norm and $\alpha>0$. 

\begin{cor}
Take $\alpha > 0$ and set $w(\cdot)\coloneqq|\cdot|^{\alpha}$ where $|\cdot|$ denotes the Euclidean norm. Consider $(\setR^n,\dist_e,\mu)$ where $\mu\df w\leb^n$. Then for any $u : \setR^n \to \setR$ two times differentiable at $0_n$ the pointwise $\AMV$ and $\SAMV$ Laplacians of $u$ in $0_n$ exist and are given by
\begin{align}
\Delta^{\dist_e}_{\mu}u(0_n) &= \frac{n+\alpha}{2n(n+\alpha+2)}\Delta u (0_n), \label{eq:amv-alpha}\\
\tilde{\Delta}^{\dist_e}_{\mu}u(0_n) &= \frac{n^2 + (2 + \alpha)n+ \alpha}{2n(n+\alpha+2)(n+2)}\Delta u (0_n), \label{eq:samv-alpha}
\end{align}
where $\Delta\df\sum_{i} \partial_{ii}$ is the classical Laplacian.
\end{cor}

\begin{proof}
The weight $w$ is continuous, hence $0_n$ is a Lebesgue point with $w(0_n)=0$. Moreover, the infinitesimal evenness of $w$ is trivially satisfied because $w$ is even. Finally the $\alpha$-homogeneity of $w$ implies that the measures $\{\nu_r\}_{r>0}$ are constantly equal to the measure
\[
\nu = \frac{n+\alpha}{\sigma_{n-1}}|\cdot|^\alpha \leb^n \measrestr B_1(0_n),
\]
hence Proposition \ref{prop:as_even_weights} applies and yields
\begin{equation}
\Delta^{\dist_e}_{\mu}u(0_n)=\frac{1}{2} \Tr (M_\nu \nabla^2u)(0_n)
\end{equation}
where we recall that $M_\nu$ is defined in \eqref{eq:second_moment_matrix}. Let us compute $M_\nu$.  If $n = 1$,  
\[
\int_{-1}^1 y^2 |y|^{\alpha} \di \nu(y) = \frac{1+\alpha}{\sigma_{0}} \int_{-1}^1 y^2 |y|^{\alpha} \di \nu(y)  = \frac{1+\alpha}{2 + \alpha +1}\, \cdot 
\]
Assume now $n>1$. Consider $i,j \in \{1,\ldots, n\}$ such that $i \neq j$.  Then the map $y \mapsto y_i y_j |y|^{\alpha}$ is odd, so
\[
\int_{B_1(0_n)} y_i y_j |y|^{\alpha} \di y =0.
\]
Moreover, the invariance of the Lebesgue measure under exchanging coordinates $\theta_i$ and $\theta_{n}$ yields
\[
\int_{B_1(0_n)} y_i^2 |y|^{\alpha} \di y = \int_{B_1(0_n)} y_{n}^2 |y|^{\alpha} \di y  =  \int_0^1 \int_{\mathbb{S}^{n-1}} r^{\alpha + n + 1} \theta_n^2 \di \sigma(\theta)\di r=\frac{1}{n+\alpha+2} \int_{\mathbb{S}^{n-1}} \theta_{n}^2 \di \sigma(\theta)
\]
where $\sigma$ is the usual surface measure on $\mathbb{S}^{n-1}$.  Since 
\[
\sum_{1\le \ell \le n} \int_{\mathbb{S}^{n-1}} \theta_{\ell}^2 \di \sigma(\theta) =  \int_{\mathbb{S}^{n-1}} |\theta|^2 \di \sigma(\theta) = \sigma_{n-1}
\]
we get
\[
\int_{\mathbb{S}^{n-1}} \theta_{n}^2 \di \sigma(\theta) = \frac{\sigma_{n-1}}{n} \, \cdot 
\]
In the end, for any $n \ge 1$,
\[
M_\nu = \frac{n+\alpha}{n(n+\alpha+2)} I_n
\]
and the result follows.

For the $\SAMV$ Laplacian, consider $\tilde{w}$ as defined in \eqref{eq:tildew}.  We want to apply Proposition \ref{prop:sym_as_even_weights}. The weight $\tilde{w}$ is even by Remark \ref{rk:evenness}.  Thanks to $\eqref{eq:tilde}$,  in order to prove \eqref{eq:prob_meas_2}, one is left with computing the limit as $r \downarrow 0$ of
\[
\frac{r^n}{\mu(B_r(x))}
\]
for any $x \neq 0$.  This follows from the Lebesgue differentiation theorem:
\[
\frac{r^n}{\mu(B_r(x))}  = \left( \omega_n \fint_{B_r(x)} w(y) \di y \right)^{-1} \to \frac{1}{|x|^\alpha \omega_n} \, \cdot 
\]
Consequently, 
\[
\tilde{\nu} = \frac{1}{2} \left( \nu + \frac{\leb^n}{\omega_n}\right).
\]
Finally, computing the second order moments w.r.t.~to $\tilde{\nu}$, simplifying and putting back into \eqref{eq:SAMV_in_0} gives \eqref{eq:samv-alpha}.
\end{proof}

\subsubsection{Separable weights}

After Proposition \ref{prop:as_even_weights} it is natural to consider the class of separable weights, defined as follows.

\begin{D}
We say that a vanishing weight $w$ of domain $\Omega$ is separable in a neighborhood of $0_n$ if there exist $r>0$ such that $B_r(0_n) \subset \Omega$ and $f \in L^1([0,r],\leb^1)$, $g \in L^1(\mathbb{S}^{n-1},\sigma)$,  such that $w(x)=f(|x|)g(x/|x|)$ for $\leb^n$-a.e.~$x \in \Omega\backslash \{0_n\}$. We say that $f$ (resp.~$g$) is the radial (resp.~angular) part of $w$.
\end{D}

\begin{prop}
 Consider $(\Omega,\dist_e,w \leb^n)$ where $w$ is a vanishing weight of domain $\Omega$ which is separable in a neighborhood of $0_n$.  Let $f$ and $g$ be the radial and angular parts of $w$, respectively.  Set $\mu=w \leb^n$ and
\[
 \nu \coloneqq (c_f \, g) \,\sigma
\]
where $\sigma $ is the normalized surface measure on the sphere $\mathbb{S}^{n-1}$ and
\[
 c_f \coloneqq \int_0^1 f(\rho) \rho^{n+1} \di \rho \in (0,+\infty).
\]
Then the following hold.
\begin{enumerate}
\item For any function $u$ two-times differentiable at $0_n$, the $\AMV$ Laplacian $\Delta_{\mu}^{\dist_e} u(0_n)$ exists and satisfies
\begin{equation}\label{eq:formula}
\Delta_{\mu}^{\dist_e} u(0_n) = \frac{1}{2}\Tr (M_\nu \nabla^2 u)(0_n)
\end{equation}
if and only if
\begin{equation}\label{eq:nascondition}
\bigg\langle \nabla u (0_n), \int_{\mathbb{S}^{n-1}} \theta g(\theta) \di \sigma(\theta) \bigg\rangle = 0.
\end{equation}

\item For any function $u$ two-times differentiable at $0_n$ such that $\nabla u (0_n)=0$,  the $\AMV$ Laplacian $\Delta_{\mu}^{\dist_e} u(0_n)$ exists and satisfies \eqref{eq:formula}.

\item If $\int_{\mathbb{S}^{n-1}} \theta g(\theta) \di \sigma(\theta)=0_n$, then for any function $u$ two-times differentiable at $0_n$, the $\AMV$ Laplacian $\Delta_{\mu}^{\dist_e} u(0_n)$ exists and satisfies \eqref{eq:formula}.
\end{enumerate}
\end{prop}

\begin{proof}
We only prove the first assertion since the others are direct consequences. Acting as in the proof of Proposition \ref{prop:as_even_weights},  we get that $\Delta_{\mu}^{\dist_e}u(0_n)$ exists and satisfies \eqref{eq:formula} if and only if
\[
\lim\limits_{r \downarrow 0} \sum_{i=1}^n \frac{\partial u}{ \partial x_i}(0_n) \frac{1}{r^2} \fint_{B_r(0_n)} x_i \di \mu(x) = 0.
\]
Thanks to \eqref{eq:change}, we know that for any $i \in \{1,\ldots,n\}$,
\[
\frac{1}{r^2 }\fint_{B_r(0_n)} x_i \di \mu(x) = \frac{1}{ r \int_{B_1(0_n)} w(ry)\di y}\left( \int_{B_1(0_n)} y_i w(ry) \di y\right)
\]
from which the separation assumption yields
\[
\frac{1}{r^2 }\fint_{B_r(0_n)} x_i \di \mu(x) =\frac{\int_0^1 f(r \rho) \rho^{n+1} \di \rho}{r \int_0^1 f(r \rho) \rho \di \rho} \frac{\int_{\mathbb{S}^{n-1}} \theta_i [g(\theta)-g(-\theta)] \di \sigma(\theta)}{\int_{\mathbb{S}^{n-1}} g(\theta) \di \sigma(\theta)} \,  \cdot
\]
Now
\[
\int_{\mathbb{S}^{n-1}} \theta_i [g(\theta)-g(-\theta)] \di \sigma(\theta) = \int_{\mathbb{S}^{n-1}} \theta_i g(\theta)\di \sigma(\theta) -  \int_{\mathbb{S}^{n-1}} \theta_i g(-\theta)\di \sigma(\theta) = 2 \int_{\mathbb{S}^{n-1}} \theta_i g(\theta)\di \sigma(\theta)
\]
where we have used the change of variable $\theta'=-\theta$ in the second integral to get the last term. Thus
\begin{align}\label{eq:comp}
\sum_{i=1}^n \frac{\partial u}{ \partial x_i}(0_n) \frac{1}{r^2} \fint_{B_r(0_n)} x_i \di \mu(x) & = \frac{\int_0^1 f(r \rho) \rho^{n+1} \di \rho}{r \int_0^1 f(r \rho) \rho \di \rho} \sum_{i=1}^n \frac{\partial u}{ \partial x_i}(0_n) \frac{2\int_{\mathbb{S}^{n-1}} \theta_i g(\theta) \di \sigma(\theta)}{\int_{\mathbb{S}^{n-1}} g(\theta) \di \sigma(\theta)} \nonumber \\
& =  \frac{\int_0^1 f(r \rho) \rho^{n+1} \di \rho}{r \int_0^1 f(r \rho) \rho \di \rho} \frac{\bigg\langle \nabla u (0_n), 2\int_{\mathbb{S}^{n-1}} \theta g(\theta)\di \sigma(\theta) \bigg\rangle}{\int_{\mathbb{S}^{n-1}} g(\theta) \di \sigma(\theta)} \, \cdot
\end{align}
This immediately shows that if \eqref{eq:nascondition} holds, then $\Delta_\mu^{\dist_e}u(0_n)$ exists and is given by \eqref{eq:AMV_in_0}.

Assume now that  \eqref{eq:nascondition} does not hold.  Then
\[
c\coloneqq\left| \bigg\langle \nabla u (0_n),  2\int_{\mathbb{S}^{n-1}} \theta g(\theta) \di \sigma(\theta)\bigg\rangle \right|>0.
\]
Let $m \in (0,1)$ be the median of the measure $f(r\rho)\di \rho$ on $[0,1]$, so that
\[
\int_0^m f(r \rho) \di \rho = \int_m^1 f(r \rho) \di \rho= \frac{1}{2}\int_0^1 f(r \rho) \di \rho.
\]
Since $f$ is non-negative and $ \int_m^1 f(r \rho)\rho^{n+1}  \di \rho \ge m^{n+1} \int_m^1 f(r \rho) \di \rho$, we get
\begin{align*}
\frac{\int_0^1 f(r \rho) \rho^{n+1} \di \rho}{r \int_0^1 f(r \rho) \rho \di \rho} \ge \frac{m^{n+1} \int_m^1 f(r \rho) \di \rho}{r \int_0^1 f(r \rho) \rho \di \rho}  = \frac{m^{n+1} \int_m^1 f(r \rho) \di \rho}{ 2 r \int_m^1 f(r \rho) \rho \di \rho} \ge \frac{m^{n+1}}{2r} \, , 
\end{align*}
where we have used $\int_m^1 f(r \rho) \rho \di \rho \le \int_m^1 f(r \rho) \di \rho$ to get the last inequality. Then from \eqref{eq:comp} we get
\[
\left| \sum_{i=1}^n \frac{\partial u}{ \partial x_i}(0_n) \frac{1}{r^2 }\fint_{B_r(0_n)} x_i \di \mu(x) \right| \ge \frac{m^{n+1}}{2r}c \to +\infty
\]
as $r \downarrow 0$, so $\Delta_\mu^{\dist_e}u(0_n)$ does not converge.
\end{proof}

When $n=2$,  the Fourier theory easily yields a concrete equivalent form of the sufficient condition given in (\textit{3}).
\begin{cor}
Assume $n = 2$ and consider $(\Omega,\dist_e,w \leb^n)$ and $u$ as in the previous proposition.  If there exist $c \ge 1$ and $\phi(\cdot) \in \mathrm{Span}(\{\cos (m\, \cdot), \sin (m \, \cdot )\}_{m \ge 2})$ such that
\[
g(e^{\bold{i} \, \cdot}) = c + \phi(\cdot),
\]
then for any function $u$ two times differentiable at $0_2$,  the $\AMV$ Laplacian $\Delta_\mu^{\dist_e}u(0_2)$  exists and is given by \eqref{eq:formula}.
\end{cor}

\begin{proof}
Let $h \in L^1(\setR)$ be $2 \pi$-periodic and such that $g(e^{\bold{i} t})=h(t)$ for any $t \in \setR$. Since
\[
\int_{\mathbb{S}^{1}} \theta g(\theta) \di \sigma(\theta)  = \int_{0}^{2\pi} e^{\bold{i} t} g(e^{\bold{i} t}) \di t = \int_{0}^{2\pi} e^{\bold{i} t} h(t) \di t
\]
then
\[
\int_{\mathbb{S}^{1}} \theta g(\theta) \di  \sigma(\theta) = 0_2 \quad \Longleftrightarrow \quad
(*) \,:\, \begin{cases}
\int_0^{2\pi} \cos(t)h(t)  \di t=0,\\
\int_0^{2\pi} \sin(t) h(t) \di t=0.
\end{cases} 
\]
Since $g$ is non-negative, so is $h$,  hence $(*)$ is equivalent to
\[
h=c + \phi
\]
where $c\ge 1$ and $\phi \in \Span(\{\cos (m \cdot), \sin (m\cdot )\}_{m \ge 2})$, thanks to the Fourier theory.
\end{proof}

\bibliographystyle{plain}
\bibliography{Minne_Tewodrose_v3}

\end{document}